\documentclass[12pt]{article}
\usepackage[dvipdf]{graphicx}
\usepackage[latin1]{inputenc}
\usepackage{latexsym}
\usepackage{amsmath,amssymb,amsfonts,amsthm}
\usepackage{xcolor}
\usepackage{mathrsfs}
\usepackage{bbm}
\usepackage{bm}
\usepackage{enumitem}

\usepackage[numbers,comma,sort]{natbib}

\definecolor{db}{RGB}{0, 0, 130}
\definecolor{wildstrawberry}{rgb}{1.0, 0.26, 0.64}

\usepackage[colorlinks=true,citecolor=db,linkcolor=black,urlcolor=blue,pdfstartview=FitH]{hyperref}

\usepackage{pdfsync}

\definecolor{rp}{rgb}{0.25, 0, 0.75}
\definecolor{dg}{rgb}{0, 0.5, 0}

\newcommand{\R}{\mathbb{R}}

\newcommand{\EE}{\mathbb{E}}

\newcommand{\PP}{\mathbb{P}}

\makeatletter
\newcommand{\customlabel}[2]{%
   \protected@write \@auxout {}{\string\newlabel {#1}{{#2}{\thepage}{#2}{#1}{}}}%
   \hypertarget{#1}{#2\hspace{-0.14cm}}
}
\makeatother

\numberwithin{equation}{section}

\newtheorem{theorem}{Theorem}[section]

\newtheorem{definition}[theorem]{Definition}

\newtheorem{lemma}[theorem]{Lemma}

\newtheorem{remark}[theorem]{Remark}

\author{ Jesus Correa \footnote{Departamento de Matem\'atica, Universidade Estadual de Campinas, Brazil. \texttt{colivera@ime.unicamp.br}.}
 \and    Christian Olivera \footnote{Departamento de Matem\'atica, Universidade Estadual de Campinas, Brazil. \texttt{colivera@ime.unicamp.br}.} \and 
}

\title{ \Large{\textbf{From particle systems  to the stochastic compressible Navier-Stokes equations of a 
barotropic fluid.}}}
\begin{document}

\maketitle

\begin{abstract}
We propose a mathematical derivation of  stochastic compressible Navier-Stokes equation.
We consider many-particle systems with a Hamiltonian dynamics supplemented by a friction term and environmental noise. Both the interaction potential and the additional friction force are supposed to be long range in comparison with the typical distance between neighboring particles.
It is shown that the empirical measures associated to the position and velocity of the system converge to the solutions of  the stochastic compressible Navier-Stokes equations of a  barotropic fluid. Moreover, we quantify the  distance  between particles and the limit in  suitable 
in Besov and Triebel-Lizorkin spaces.  
\end{abstract}

\date{}

\maketitle

\noindent \textit{ {\bf Key words and phrases:} 
 Stochastic Compressible  Navier-Stokes system, Particle systems,  Ito-Kunita-Wentzell formula, Besov and Triebel-Lizorkin spaces.}

\vspace{0.3cm} \noindent {\bf MSC2010 subject classification:} 60H15, 
 35R60, 60H30, 35Q31
. 

%
%
%
%
%
%
%

Data sharing not applicable to this article as no datasets were generated or analysed during the current study
\section{Introduction}

In this paper we consider stochastic particle system on $\R^{d}, d\geq 2$ , with moderate interaction(\cite{Oes2} and \cite{Oes4}   ) and the 
so-called environmental noise :

\begin{align}\label{Nsrb}
\left\{
\begin{array}{lcl}
\frac{dX_{t}^{k,N}}{dt}&=& V_{t}^{k,N}\\[0.2cm]
d V_{t}^{k,N}&=&-\frac{1}{N}\displaystyle \sum_{l=1 \atop l \neq k}^{N} \nabla \phi_{N}\left(X_{t}^{k,N}-X_{t}^{l,N}\right)dt\\
&&-\frac{1}{N}\displaystyle \sum_{l=1 \atop l \neq k}^{N} \zeta_{N}\left(X_{t}^{k,N}-X_{t}^{l,N}\right)\left(V_{t}^{k,N}-V_{t}^{l,N}\right)dt\\
&&+\sigma\left(X_{t}^{k,N} \right)  \frac{dX^{k,N}}{dt}  \circ dB_{t} \qquad k=1, \ldots, N
\end{array}
\right.
\end{align}

\noindent where   $\{(B_{t}^{q})_{t\in[0,T]}, \; q=1,\cdots,d\}$ is a  standard $\R^d$-valued Brownian motions
  defined on a filtered probability space $\left(\Omega,\mathcal{F},\mathcal{F}_{t},\PP\right) $, and the integration in the sense of Stratonovich.  The interaction potential $\phi_{N}$, see formula   (\ref{defphi}), 
and the matrix-value function $\zeta_{N}$, see formula  (\ref{defzetamatriz}), which determines the friction between different particles, depend on $N$.

In the classical studies of continuum interacting particle systems, where interactions are modulated by a potential, one usually takes
$\phi_{N}=N^{\beta} \phi_{1}(N^{\beta/d}) $
for some smooth function $\phi_{1}$. The case $\beta = 0$ is called mean-field, since all particles interact with each
other at any given time. The case $\beta\in  (0, 1)$ is called moderate, as not all particles interact at
any given time, nevertheless such number is diverging with $N$. The case $\beta = 1$ is called local, as
one would expect that in a neighborhood of radius $N^{-\frac{1}{d}}$ , only a constant number of particles
interact. Of course, here we are assuming that particles are relatively homogeneously distributed
in space at all times down to the microscopic scale. The matrix $\zeta_{N}$  determines the friction term 
between particles, see \cite{Oes3} for physical  arguments.

Our main goal is to derive  rigorously the macroscopic collective models  governing the evolution of the particle system (\ref{Nsrb}) as the number of particles goes to infinity. Therefore, we investigate the empirical processes 

\begin{eqnarray}\label{emppXV}
 S_{t}^{N}:=\dfrac{1}{N}\sum_{k=1}^{N}\delta_{X_{t}^{k,N}}\qquad  V_{t}^{N}:=\dfrac{1}{N}\sum_{k=1}^{N}V_{t}^{k,N}\delta_{X_{t}^{k,N}}\qquad k=1,\cdots, N
\end{eqnarray}

where  $\delta_{a}$, denotes the
Dirac measure at $a$, for any test functions $f\in C_{b}(\R^{d})$ we can define the random  measures
$\displaystyle S_{t}^{N}= \dfrac{1}{N}\sum_{k=1}^{N}   f(X_{t}^{k,N})$ and  $\displaystyle V_{t}^{N}= \dfrac{1}{N}\sum_{k=1}^{N}V_{t}^{k,N} f(X_{t}^{k,N})$. The random measures $S_{t}^{N}$ and $V_{t}^{N}$ determine the distribution
of the positions and the velocities in the $Nth$ system. 
We shall show that $S_{t}^{N}$ and $V_{t}^{N}$ converge as $N\rightarrow\infty$ to 
$\rho$ and $\rho \upsilon$ where $(\rho,\upsilon)$ is the solution of 
a  stochastic systems  for a compressible, viscous barotropic flow. More precisely, we show the convergence to  the following  stochastic systems 
of the continuity equation and Navier-Stokes equation 
\begin{align}\label{NV}
\left\{
\begin{array}{lc}
d\varrho=-\operatorname{div}_{x}(\varrho\upsilon)dt
\\[0.3cm]
d(\varrho\upsilon_{q})=\Bigl(-\dfrac{1}{2}\nabla^{q}(\varrho^{2})-\operatorname{div}_{x}(\varrho\upsilon_{q}\upsilon)+\frac{1}{2}\nabla^{q}\left(\varrho^{2}\operatorname{div}_{x} \upsilon\right)\Bigr.\\
\qquad\quad\qquad\Bigl.+\displaystyle\frac{1}{2}\sum_{i=1}^{d}\nabla^{i}\left(\varrho^{2}\left[\nabla^{i} \upsilon_{q}+\nabla^{q} \upsilon_{i}\right]\right)\Bigr)dt+ \varrho \sigma_{q}(x)\upsilon_{q}\circ dB_{t}^{q}\quad q=1,\cdots, d.
\end{array}
\right.
\end{align}

Applying   to  $\varrho\upsilon_{q}$  the  Ito formula for the product of  semimartingale 
and  using the equation that   verifies  $\varrho$,  we can show  that the system  (\ref{NV}) is equivalent to

\begin{align}\label{NVSst}
\left\{
\begin{array}{lc}
d\varrho=-\operatorname{div}_{x}(\varrho\upsilon)dt
\\[0.3cm]
d\upsilon_{q}=\left(-v \cdot \nabla \upsilon_{q}-\nabla^{q} \varrho+\frac{1}{2 \varrho} \nabla^{q}\left(\varrho^{2} \operatorname{div}_{x} v\right)+\frac{1}{2\varrho} \displaystyle\sum_{i=1}^{d} \nabla^{i}\left(\varrho^{2}\left[\nabla^{i} \upsilon_{q}+\nabla^{q} \upsilon_{i}\right]\right)\right)dt\\
\qquad\qquad\qquad+ \sigma_{q}( x) \upsilon_{q} \circ dB_{t}^{q}\qquad q=1,\cdots, d
\end{array}
\right.
\end{align}

where the pressure is $p=\frac{1}{2}\varrho^{2}.$  \\

The problem of deriving the  deterministic and stochastic hydrodynamic equations has been studied over the last century 
by various approaches, and the aim of this paper is to  present the first macroscopic derivation of the stochastic compressible Navier  equation. Our proof involves several tools, for example :   Ito-Kunita-Wentzell formula, stochastic  maximal inequality, commutator estimates, Taylor expansion   and some ideas in  the paper \cite{Oes3} where the  same equation   without noise  is considered.  We use several elements of the proof it gives in   \cite{Oes3}, however in that paper the author  assume the existence of a  function $W^{N}$ such that  $W^{N}(X_{t}^{i,N})=V_{t}^{i,N}$,  
the existence of $W^N$ is simply the mono-kinetic regime, that means,  if particles share the same position, then their velocities should be the same as well. The existence of such function is used throughout of the paper \cite{Oes3}, we abolished the use of this function in our proof.

 The mathematical analysis of the coupled system (\ref{NVSst})  has received much attention in the past years, we refer to the reader to \cite{Breit}, \cite{Smith}, \cite{Smith2}, \cite{Tor} and \cite{Tor2}.

\paragraph{ Some related Works.}
The literature on propagation of chaos and fluid dynamic equations consists of a number of works,
including   \cite{FlandoliOliveraSimon, FHM, Gui, JabinWang,Marra,  Marchioro, Meleard3, Osada}. The literature on particle with  common noise   remains limited ,  for systems with uniformly Lipschitz interaction coefficients \cite{Cogui} established conditional propagation of chaos,  the entropy method has recently been explored for systems with common noise, as shown in \cite{Shao} for incompressible  the Navier-Stokes equations, \cite{Chen2} for the Hegselmann-Krause model, and \cite{Niko} for mean-field systems with bounded kernels.For more results , see \cite{Maurelli}, \cite{Cogui}, \cite{correa}  and  \cite{Kotelenez}.

In the following we will focus only on the results of moderately interacting systems  :  in  \cite{Oelschlager84} Oelschlager introduced and studied moderately interacting particle systems which are used to obtain local non-linear partial differential equations. The article \cite{Oelschlager84}  is part of a series of works by the author on this subject (see also \cite{Oes2,Oes4}), the convergence results was improved  in \cite{JourdainMeleard} and \cite{Meleard}.  In the recent paper \cite{FlandoliLeimbachOlivera} was  developed a semigroup approach which enables them to show uniform convergence of mollified empirical measures,  see \cite{FlandoliLeocata,Leocata, Live, FlandoliOliveraSimon,Simon} for further applications of
this method.  Recently, in \cite{Pisa} and \cite{ORT} the authors developed new quantitative  estimates for classes of moderately interacting particle systems,  by using a semi-group approach, this approach 
was improvements  in \cite{Hao},  \cite{Knorst} and   \cite{Simon2}. About more advances in  moderate particle systems,  see for instance  \cite{Ansgar}, \cite{Chen}, \cite{correa},  \cite{Ansgar2}, and \cite{Steve}.

Finally,  we want to emphasize that our particle systems is very close to particle method (SPH)
 to approximated PDE. This method was first introduced in  \cite{Lucy} and \cite{Mona}  to simulate fluids in astrophysics,and it became popular due toits robustness and easy implementation. Due its simplicity, the method has many applications in multiple fields, for example in hydrodynamics, bioengineering or astrophysics. A detailed list of applications can be found in \cite{Liu} and \cite{Sha}.  However there are a few papers which already deal with the convergence  of the SPH method, see \cite{Dil}, \cite{Dil2}, \cite{Franz},  \cite{Oes} and \cite{Oes3}.

\subsection{Heuristic deduction}

By change rule, for any function  $f\in C_{b}^{1}(\R^{d})$, we have

\begin{equation}\label{f(XN)N-S}
d(f(X_{t}^{k,N}))=\nabla f(X_{t}^{k,N})d(X_{t}^{k,N})=\left(\nabla f(X_{t}^{k,N})\cdot V_{t}^{k,N}\right) dt.
\end{equation}

Then we have
\begin{eqnarray}\label{eq1d.N-S}
\langle S_{t}^{N},f\rangle&=&\langle S_{0}^{N},f\rangle+\int_{0}^{t}\langle V_{t}^{N},\nabla f\rangle ds.
\end{eqnarray}
On other hand, we obtain 

\begin{align*}
\left( S_{t}^{N}\ast \phi_{N}\right)(x)&=\left(\dfrac{1}{N}\sum_{l=1}^{N}\delta_{X_{t}^{l,N}}\ast\phi_{N}\right)(x)\\
&=\dfrac{1}{N}\sum_{l=1}^{N}\phi_{N}(x-X_{t}^{l,N}).
\end{align*}

We observe that the system \eqref{Nsrb} is equivalent to the  system:

\begin{align}\label{NSVN}
\left\{
\begin{array}{lcl}
d X_{t}^{k,N}&=& V^{k,N} dt\\[0.3cm]
d\left(V_{t}^{k,N}\right)&=&-\nabla\left(\mathbf  S_{t}^{N}\ast \phi_{N}\right)\left(X_{t}^{k,N}\right)dt\\[0.3cm]
&&-\frac{1}{N}\displaystyle \sum_{l=1 \atop l \neq k}^{N} \zeta_{N}\left(X_{t}^{k,N}-X_{t}^{l,N}\right)\left(V_{t}^{k,N}-V_{t}^{l,N}\right)dt\\[0.3cm]
&&+\sigma (X_{t}^{k,N})  V_{t}^{k,N} \circ d B_{t}\qquad k=1,\cdots, N.
\end{array}
\right.
\end{align}

Applying the Ito formula  for  the product  to  \eqref{NSVN} and  \eqref{f(XN)N-S} we have 
\begin{align*}
d(V_{t,q}^{k,N}f(X_{t}^{k,N}))=&f(X_{t}^{k,N})\circ d(V_{t,q}^{k,N})+V_{t,q}^{k,N}\circ d(f(X_{t}^{k,N}))\\
    =&-f(X_{t}^{k,N})\nabla^{q}\left( S_{t}^{N}\ast\phi_{N}\right)(X_{t}^{k,N})))dt\\
    &-\frac{1}{N} \sum_{l=1 \atop l \neq k}^{N} f(X_{t}^{k,N})\zeta_{N,q*}\left(X_{t}^{k,N}-X_{t}^{l,N}\right)\left(V_{t}^{k,N}-V_{t}^{l,N}\right)dt\\
    &+f(X_{t}^{k,N})\sigma (X_{t}^{k,N})  V_{t,q}^{k,N} \circ d B_{t}^{q}+ V_{t,q}^{k,N}\nabla f(X_{t}^{k,N})\cdot V_{t}^{k,N} dt
\end{align*}
or , equivalently,
\begin{align*}
&f(X_{t}^{k,N})V_{t,q}^{k,N}-f(X_{0}^{k,N})V_{0,q}{k,N}=-\int_{0}^{t}f(X_{s}^{k,N})\nabla^{j}\left( S_{t}^{N}\ast\phi_{N}\right)(X_{s}^{k,N})ds\\
&-\int_{0}^{t}\frac{1}{N} \sum_{l=1 \atop l \neq k}^{N} f(X_{s}^{k,N})\zeta_{N,q*}\left(X_{s}^{k,N}-X_{s}^{l,N}\right)\left(V_{s}^{k,N}-V_{s}^{l,N}\right)ds\\
& \quad+\int_{0}^{t}f(X_{s}^{k,N})\sigma (X_{s}^{k,N})  V_{s,q}^{k,N} \circ d B_{s}^{q}+\int_{0}^{t}V_{s,q}^{k,N}\nabla f(X_{s}^{k,N})\cdot V_{s}^{k,N}ds.
\end{align*}

Thus we have 
\begin{align}\label{EnuN-S}
&\frac{1}{N}\sum_{k=1}^{N}f(X_{t}^{k,N})V_{t,q}^{k,N}-f(X_{0}^{k,N})V_{0,q}{k,N}\nonumber\\
=&-\int_{0}^{t}\frac{1}{N}\sum_{k=1}^{N}f(X_{s}^{k,N})\nabla^{j}\left( S_{t}^{N}\ast\phi_{N}\right)(X_{s}^{k,N})ds\nonumber\\
&-\int_{0}^{t}\frac{1}{N^{2}}\sum_{k=1}^{N}\sum_{l=1 \atop l \neq k}^{N} f(X_{s}^{k,N})\zeta_{N,q*}\left(X_{s}^{k,N}-X_{s}^{l,N}\right)\left(V_{s}^{k,N}-V_{s}^{l,N}\right)ds\nonumber\\
&+\int_{0}^{t}\frac{1}{N}\sum_{k=1}^{N}f(X_{s}^{k,N})\sigma (X_{s}^{k,N})  V_{s,q}^{k,N} \circ d B_{s}^{q}+\int_{0}^{t}\frac{1}{N}\sum_{k=1}^{N}V_{s,q}^{k,N}\nabla f(X_{s}^{k,N})\cdot V_{s}^{k,N}ds.
\end{align}

Let us now introduce formally  an $\R^{d}$-valued  function $\nu^{N}$
which satisfies $\nu^{N}(X_{t}^{k,N}, t) = V_{t}^{k,N}$,$ k = 1,\cdots,N$. Obviously, such a
function can be defined only if different particles cannot occupy the same position
at the same instant. From \eqref{EnuN-S} we deduce

\begin{align}\label{eq2d.N-S}
\langle S_{t}^{N},f \nu_{q}^{N}(\cdot,t)\rangle=&\langle S_{0}^{N},f\nu_{q}^{N}(\cdot,0)  \rangle -
\int_{0}^{t}\langle S_{s}^{N},f\nabla^{q} \left(S_{s}^{N}\ast\phi_{N}\right)\rangle ds\nonumber\\
&-\int_{0}^{t}\langle S_{s}^{N},f(\cdot)\int_{\R^{d}}\zeta_{N,q*}\left(\cdot-y\right)\left(\nu^{N}(\cdot,s)-\nu^{N}(y,s)dS_{s}^{N}(y)\right)\rangle ds\nonumber\\
&+\int_{0}^{t}\langle S_{s}^{N},\nu_{q}^{N}(\cdot,s) \nabla f\cdot \nu^{N}(\cdot,s)\rangle ds+\int_{0}^{t}\langle S_{s}^{N},f \sigma(\cdot)\nu_{q}^{N}   \rangle\circ dB_{s}^{q}\\
&\qquad\qquad\qquad\qquad\qquad\qquad\qquad\qquad\qquad\quad q=1,\cdots,d.\nonumber
\end{align}

Formally, we now assume that for any fixed t the limits
$ \lim_{N\rightarrow \infty} S_{t}^{N}=\rho(\cdot,t)$ , $ \lim_{N\rightarrow \infty} \nu^{N}(\cdot,t)=\upsilon(\cdot,t)$
exist in some sense, by  Taylor expansion, in space variable, we obtain

\begin{align*}
&\lim _{N \rightarrow \infty} \int_{\mathbb{R}^d}\sum_{q^{\prime}=1}^a \zeta_{N, qq^{\prime}}(x-y)\left(\nu_{q^{\prime}}^{N}(x, s)-\nu_{q^{\prime}}^{N}(y, s)\right)dS_{t}^{N}\nonumber\\
=&\lim _{N \rightarrow \infty} \int_{\mathbb{R}^d} \rho(y, s) \sum_{q^{\prime}=1}^d \zeta_{N, qq^{\prime}}(x-y)\left(\upsilon_{q^{\prime}}(x, s)-\upsilon_{q^{\prime}}(y, s)\right)dy\nonumber \\
=&\lim _{N \rightarrow \infty} \int_{\mathbb{R}^d}  \sum_{q^{\prime}=1}^d\left\{\rho(x, s)+\sum_{i=1}^d\left(y_i-x_i\right) \nabla^{j} \rho(x, s)\right\}\nonumber\\ &\qquad\times\dfrac{N^{\gamma(d+4)/d}}{(2\pi)^{d/2}}\left(x_{q}-y_{q}\right)\left(x_{q^{\prime}}-y_{q^{\prime}}\right)e^{-(N^{\gamma/d}|x-y|)^{2}/2}\nonumber\\
&\qquad\times \left\{\sum_{j=1}^d\left(x_j-y_j\right) \nabla^{j} \upsilon_{q^{\prime}}(x, s)-\frac{1}{2} \sum_{i, j=1}^d\left(x_i-y_i\right)\left(x_j-y_j\right) \nabla^{j} \nabla^{j} \upsilon_{q^{\prime}}(x, s)\right\}dy\nonumber \\
=&-2\rho(x, s) \sum_{i=1} \nabla^{j} \nabla^{j} \upsilon_{q}(x, s)- \rho(x, s) \nabla^{q} \nabla^{q} \upsilon_{q}(x, s) -\rho(x, s) \sum_{\substack{i=1 \\
i \neq q}}^d \nabla^{q} \nabla^{j} v_i(x, s)\\
&- \sum_{i=1}^d \nabla^{j} \rho(x, s) \nabla^{j} \upsilon_{q}(x, s)-2\nabla^{q} \rho(x, s) \nabla^{q} \upsilon_{q}(x, s)\\
&-\sum_{\substack{i=1 \\
i \neq q}}^d\left(\nabla^{q} \rho(x, s) \nabla^{j} v_i(x, s)+\nabla^{j} \rho(x, s) \nabla^{q} v_i(x, s)\right) \\
=&  -\frac{1}{2\rho(x, s)} \nabla^{q}\left(\rho(x, s)^2 \sum_{i=1}^d \nabla^{j} v_i(x, s)\right)\nonumber\\
&-\frac{1}{2\rho(x, s)} \sum_{i=1}^d \nabla^{j}\left(\rho(x, s)^2\left[\nabla^{j} \upsilon_{q}(x, s)+\nabla^{q} v_i(x, s)\right]\right),
\end{align*}

where the constant coefficients of the previous equation are determined by

\begin{align*}
\dfrac{N^{\gamma(d+4)/d}}{(2\pi)^{d/2}}\int_{\R^{d}}x^{\alpha}x_{q}x_{q^{\prime}}& e^{-(N^{\gamma/d}|x|)^{2}/2}dx\\
&=N^{\gamma(d+4)/d}i^{-(|\alpha|+2)}\dfrac{\partial^{\alpha}}{\xi^{\alpha}}\nabla^{q}\nabla^{q^{\prime}}\left(\dfrac{1}{N^{\gamma}}e^{-(N^{-\gamma/d}|\xi|)^{2}/2}\right)(0)\\
&=(-1)^{|\alpha|+1}i^{|\alpha|}N^{4\gamma/d}\frac{\partial^{\alpha}}{\xi^{\alpha}}\nabla^{q}\nabla^{q^{\prime}}\left(e^{-(N^{-\gamma/d}|\xi|)^{2}/2}\right)(0),
\end{align*}

Thus, by taking the limit when $N\to\infty$ in the equations \eqref{eq1d.N-S} and \eqref{eq2d.N-S} as $N\to\infty$, we can formally deduce  the system of equations (\ref{NV}).

\section{ Definitions,  preliminaries and hypothesis.}

\subsection{Space of functions }

Let us first define a dyadic partition of unity, as follows: we consider two 
$C_{0}^{\infty}(\R^{d})$-functions $\chi$ 
 and $\varphi$  which take values in $[0,1]$ and satisfy the following: there exists $\lambda \in (1,\sqrt 2)$ such that
\[
\mathrm{Supp}\; \chi=\left\{   |\xi |\leq\lambda  \right\} \qquad \text{and} \qquad \mathrm{Supp}  \ \varphi=\left\{  \frac{1}{\lambda} \leq|\xi|\leq\lambda  \right\}.
\]
Moreover, with the following notations,  
\[
\varphi_{-1}(\xi):=\chi(\xi), \qquad
\varphi_{i}(\xi):=\varphi(2^{-i}\xi), \ \text{for any } i\geq 0,
\]
the sequence $\{\varphi_i\}$ satisfies
\begin{align*}
& \mathrm{Supp} \ \varphi_{i}\cap \mathrm{Supp} \ \varphi_{j}=\emptyset	 \quad \text{ if } \ |i-j|>1,  \\
& \sum_{i\geq-1} \varphi_{i}(\xi)=1, \quad \text{ for any } 
x \in\R^{d}.
\end{align*}

 Take a fixed dyadic partition of unity $\{ \varphi_{i} \}$ with its inverse Fourier transforms 
$  \{  \check{\varphi}_{i}   \}$. 
For $u\in \mathcal{S}'(\R^{d})$
, the non-homogeneous dyadic  blocks are defined  as
\[
\Delta_{i}\equiv 0, \quad \text{ if } i< -1, \qquad \text{and} \qquad 
\Delta_{i}u=\check{\varphi}_{i}\ast u, \quad \text{ if  }  i\geq -1.
\]
The partial sum of dyadic blocks is defined as a non-homogeneous low frequency cut-off
operator:
\[
S_{j}:=\sum_{i\leq j-1} \Delta_{i}. 
\]

Now, we introduce the Triebel-Lizorkin and Besov Space.

\begin{definition}  Let $s\in \R$ and $1 <r\leq\infty$. 

\begin{enumerate}

\item  If $0< p< \infty$, then

\[
F_{p,r}^{s}= \left\{  f\in \mathcal{S}^{\prime}(\R^{d}) : \    \big\| \| (2^{js} \Delta_{j}f)_{j \in \mathbb{Z} } \|_{l^r(\mathbb{Z})}  \big\|_{\mathit{L}^{p}(\R^{d})}< \infty \right\}.
\]

\item If $0< p\leq \infty$, then

\[
B_{p,r}^{s}= \left\{  f\in \mathcal{S}^{\prime}(\R^{d}) : \  \big\| (2^{js}\| \Delta_{j}f \|_{L^p} )_{j \in \mathbb{Z} } \big\|_{\mathit{l}^{r}(\mathbb{Z})}< \infty \right\},
\]
\end{enumerate}

\end{definition} 

where $l^{r}(\mathbb{Z}^{d})= \left\{ \{ a_{n}\}_{n\in \mathbb{Z}^{d} } \ with \  a_{n}\in \mathbb{R} \ 
 : \| \{ a_{n}  \}_{n\in \mathbb{Z}^{d}} \}\|^{p}= \sum_{n\in \mathbb{Z}^{d} } | a_{n} |^{p}< \infty     \right\} $.
 
The spaces $F_{p,r}^{s}$  and $B_{p,r}^{s}$are independent of the dyadic partition  chosen, see \cite{Triebel}. We denoted 
$E_{p,r}^{s}=F_{p,r}^{s}$ or $E_{p,r}^{s}=B_{p,r}^{s}$ with $q\geq 2$. We recall the Sobolev embedding

\begin{equation}\label{Sobolev}
E_{p,r}^{s}\subset C_{b}^{s-\frac{d}{p}}
\end{equation}

if $s> \frac{d}{p}$, see p. 203 of \cite{Triebel}.

We denoted $C_{b}(\R^{d}; \R^{d})$  is the space of bounded continuous $\R^{d}$-valued functions on $\R^{d}$.

\subsection{Stochastic calculus }
We recall to help the intuition, the following definitions 

$$
\begin{aligned}
\text{Ito:}&  \ \int_{0}^{t} X_s dB_s=
\lim_{n    \rightarrow \infty}   \sum_{t_i\in \pi_n, t_i\leq t}  X_{t_i}    (B_{t_{i+1} \wedge t} - B_{t_i}),
\\[5pt]
\text{Stratonovich:}&  \ \int_{0}^{t} X_s  \circ dB_s=
\lim_{n    \rightarrow \infty}   \sum_{t_i\in \pi_n, t_i\leq t} \frac{ (X_{t_i \wedge t   } + X_{t_i} ) }{2} (B_{t_{i+1} \wedge t} - B_{t_i}),
\\[5pt]
\text{Covariation:}& \ [X, Y ]_t =
\lim_{n    \rightarrow \infty}   \sum_{t_i\in \pi_n, t_i\leq t} (X_{t_i \wedge t   } - X_{t_i} )  (Y_{t_{i+1} \wedge t} - Y_{t_i}),
\end{aligned}
$$
where $\pi_n$ is a sequence of finite partitions of $ [0, T ]$ with size $ |\pi_n| \rightarrow 0$ and
elements $0 = t_0 < t_1 < \ldots  $. The limits are in probability, uniformly in time
on compact intervals. Details about these facts can be found in Kunita 
\cite{Ku2}. 

 Now, we recall  the Ito-Kunita-Wentzell formula, see  Theorem 8.3 of \cite{Ku2}. 
We consider $X(t,x, \omega)$ be a continuous  $C^{3}$-process and 
continuous  $C^{2}$-semimartingale , for $x\in \R^{d}$, and $t\in [0,T]$ of the form 

\[
  X(t,x,\omega)=X_{0}(x) + \int_{0}^{t} f(s,x,\omega) ds  + \int_{0}^{t} g(s,x,\omega) \circ dB_s. 
 \]

If $Y(t,\omega)$   is  a continuous semimartingale, then $X(t,Y_{t},\omega)$
is a continuous semimartingale and the following formula holds

\[
  X(t,Y_{t},\omega)=X_{0}(Y_{0}) + \int_{0}^{t} f(s,Y_{s},\omega) ds  + \int_{0}^{t} g(s,Y_{s},\omega) \circ dB_s
\]

\[
+ \int_{0}^{t} (\nabla_{x}X)(s,Y_{s},\omega)\circ dY_{s}.
\]

\subsection{Definition of solution }

  \begin{definition}
	Let  $\varrho_{0}\in E_{p,r}^{s}\cap L^{1}(\R^{d})$ with $\varrho_{0}>0$, and $\upsilon_{0}\in  E_{p,r}^{s}$ with   $s>\frac{d}{p}+2$, $1<p<\infty$, 
$1<r< \infty$. Then  $(\varrho, \upsilon)$   is called a solution of (\ref{NVSst}) if the following conditions are satisfied

\begin{enumerate}
    \item[(i)] $\left(\varrho, \upsilon\right)$ is a $E_{p,r}^{s} \cap L^{1}(\R^{d}) \times E_{p,r}^{s} $-valued right-continuous progressively measurable process and $\rho>0$;
\item[(ii)] $\tau$ is a stopping time with respect to $\left(\mathfrak{F}_{t}\right)$ such that $\mathbb{P}$-a.s.
$$
\tau(w)=\lim _{m \rightarrow \infty} \tau_{m}(w)
$$
where
$$
\tau_{m}=\inf \left\{0 \leq t<\infty:\left\|\left(\varrho, \upsilon\right)(t)\right\|_{E_{p,r}^{s}} \geq m\right\}
$$
with the convention that $\tau_{m}=\infty$ if the set above is empty;
\item[(iii)] there holds $\mathbb{P}$-a.s.
$$
   (\varrho, \upsilon)   \in C\left([0, \tau_{m}(\omega)] ; E_{p,r}^{s} \cap L^{1}(\R^{d})  \times   E_{p,r}^{s} \right)
$$
as well as
\begin{align*}
\varrho\left(t \wedge \tau_{m}\right)=&\varrho_{0}-\int_{0}^{t \wedge \tau_{m}} \operatorname{div}_{x}(\varrho\upsilon)\mathrm{d} s, \\
\upsilon_{q}\left(t \wedge \tau_{m}\right)=&\upsilon_{q,0}-\int_{0}^{t \wedge \tau_{m}} \left(\frac{1}{\varrho} \nabla^{q}p+\upsilon  \nabla  \upsilon_{q} \right)\mathrm{d} s\\
&+\int_{0}^{t \wedge \tau_{m}}\frac{1}{2\varrho} \sum_{i=1}^{d} \nabla^{i}\left(\varrho^{2}\left[\nabla^{i} \upsilon_{q}+\nabla^{q} \upsilon_{i}\right]\right)ds 
+\int_{0}^{t \wedge \tau_{m}} \sigma_{q} \upsilon_{q} \circ dB_{t}^{q}
\end{align*}
for all $q=1,\cdots, d$, all $t \in[0, T]$ and all $m\geq 1$.
\end{enumerate}

  \end{definition}

\begin{remark} On existence and uniqueness  of system of the form  (\ref{NVSst})
we refer to the reader to \cite{Breit}, \cite{Smith}, \cite{Smith2}, \cite{Tor} and \cite{Tor2}. 
\end{remark}

\begin{remark} We recall that by the method of characteristics $\varrho= \varrho_{0}(X_{t}^{-1})$, where $X_{t}$ is the flow 
with drift $\upsilon$, that is, $X_{t}=x+ \int_{0}^{t}\upsilon(t,X_{t}) \  dt$. Therefore,  if $\varrho_{0}>0$  then $\varrho>0$. 
\end{remark}

\begin{remark} We observe that by Sobolev embedding \\  $(\varrho, \upsilon)   \in C\left([0, \tau_{m}(\omega)] ;  C_{b}^{2} \cap L^{1}(\R^{d})  \times   C_{b}^{2}(\R^{d}) \right)$. Then the solutions are classical in the space variable.  The existence of global solutions in time  and explosion criterion are   difficult problems even in the deterministic case, see for instance \cite{Hoff}, \cite{Hoff2} and \cite{Hoff3}. 
The stochastic case will be a topic of future research.
\end{remark}

\subsection{ Technical hypothesis}

Next, we suppose that $\phi_{1}$ can be written as a convolution product 

 $$\phi_{1}=\phi_{1}^{r}\ast\phi_{1}^{r}$$

 where  $\phi_{1}^{r}\in C_{b}^{2}(\R^{d})$  and it is symmetric probability density.

\begin{definition}
For all  $q\in\{1,\cdots,d\}$  we define the function 
\begin{eqnarray*}
U_{1;\alpha}^{q}:\R^{d}&\longrightarrow&\R\\
x&\longmapsto&(-1)^{1+|\alpha|} \dfrac{x^{\alpha}}{\alpha!} \nabla^{q}\phi_{1}^{r}(x)
\end{eqnarray*}
where  $0\leq|\alpha|\leq L+1$ with $L:=\left[\frac{d+2}{2}\right]$.
\end{definition}

We assume that functions $\phi_{1}^{r}$ and $U_{1;\alpha}^{q}$ satisfy

\begin{eqnarray}\label{c1}
|\phi_{1}^{r}(x)|\leq\frac{C}{1+|x|^{d+2}}\qquad |x|\geq1,
\end{eqnarray}
\begin{eqnarray}\label{gradiendtetranforfouphi1r}
    \nabla\widetilde{\phi_{1}^{r}}\in L^{\infty}(\R^{d};\R^{d}),
\end{eqnarray}
\begin{eqnarray}\label{limite inferiorphi1r}
    exp(-|\lambda|^{2}/4)\leq C|\widetilde{\phi_{1}^{r}}(\lambda)|,
\end{eqnarray}
\begin{eqnarray}\label{cotawildeu}
|\widetilde{U^{q}_{1;\alpha}}(\lambda)|\leqq C|\widetilde{\phi_{1}^{r}}(\lambda)|,\qquad\text{para } 0\leq|\alpha|\leq L
\end{eqnarray}
\begin{eqnarray}\label{cotauj}
|U^{q}_{1;\alpha}(x)|\leqq C \left(\frac{1}{1+|x|^{d+1}}\right)^{1/2},\qquad\text{para } |\alpha|=L+1.\
\end{eqnarray}

Here, the notation 

\[
\widetilde{f}(\lambda)= (2\phi)^{-d/2} \int_{\R^{d}} f(x) \exp{(i\lambda x)} \ dx,
\]

for the Fourier transform. 

We set

\begin{eqnarray}\label{defphi}
  \phi_N(x) =N^\beta \phi_1\left(N^{\beta / d} x\right), \quad 0<\beta<1
  \end{eqnarray}
\begin{eqnarray}\label{defzetamatriz}
\zeta_{N, i j}(x)= x_{i} x_{j} N^{4 \gamma / d} \psi_N(x), i, j=1, \ldots, d,, \quad 0<\gamma<\frac{2 \beta}{3 d+8} , 
\end{eqnarray}

where 
\begin{equation}\label{defexp}
\psi_{1}(x)=(2\pi)^{-d/2}e^{-|x|^{2}/2}\quad\text{and}\quad\psi_N(x)=N^\gamma \psi_1\left(N^{\gamma / d} x\right).
 \end{equation}

   On the existence of the function  $\phi_{1}(x)$ we refer to 
   \cite{Oes3}. We also assume that $\sigma\in C_{b}^{1}( \R^{d}, \R^{d})$.

\section{Result}

First we present a simple lemma. 

\begin{lemma}
We assume \eqref{c1}, then there exist $C^{\prime}>0$ such that
    \begin{eqnarray}\label{cotafNbeta}
\Vert f-(f\ast\phi_{N}^{r})\Vert_{\infty}&\leq&C^{\prime}N^{-\beta/d}\Vert\nabla f\Vert_{\infty}\qquad\forall f\in C_{b}^{1}(\R^{d}).
\end{eqnarray}
\end{lemma}

\begin{proof} See \cite{correa}.
\end{proof}

We define the  empirical energy 
 
\begin{equation}\label{QN}
Q_{t}^{N}:=\frac{1}{N}\sum_{k=1}^{N}|V_{t}^{k,N}-\upsilon(X_{t}^{k,N},t)|^{2} +\Vert S_{t}^{N}\ast\phi_{N}^{r}-\varrho(.,t)\Vert_{L^{2}(\R^{d})}^{2}.
\end{equation}

We assume that

\begin{equation}\label{paramdelta}
    \delta \in\left(\frac{\gamma}{d}(d+4), \frac{\beta}{d} \wedge \frac{2}{d}(\beta-\gamma(d+ 2 ))\right).
\end{equation}

\begin{theorem}\label{tomeuler}
Let $s\geq  ( \delta d/2\lambda ) +4+d/p$ and $\alpha>\frac{d}{2}+1$. We assume  \eqref{c1}-\eqref{defexp}
and

\begin{equation}\label{condi,ini.Q0N}
    \lim_{N\to\infty}N^{2\delta}\mathbb{E} \big( Q_{0}^{N} \big)^{2}=0.
\end{equation}

Then we have

\begin{equation*}
\lim _{N \rightarrow \infty}  N^{2\delta} \mathbb{E} \big( \sup_{[0, T]} Q_{t\wedge  \tau_{m}  }^{N}  \big) ^{2}  =0,
\end{equation*}

\begin{equation*}
\lim _{N \rightarrow \infty}  N^{2\delta} \mathbb{E} \sup_{[0, T]} \|S_{t\wedge  \tau_{m} }^{N}- \varrho_{t \wedge  \tau_{m} } \|_{E_{2,\tilde{r}}^{-\alpha}}^{2}=0
\end{equation*}

and

\begin{equation*}
\lim _{N \rightarrow \infty}  N^{2\delta} \mathbb{E} \sup_{[0, T]} \|V_{t\wedge  \tau_{m} }^{N}- (\varrho \upsilon)_{t \wedge  \tau_{m} } \|_{E_{2,\tilde{r}}^{-\alpha}}^{2},
=0. \end{equation*}

where $\frac{1}{r}+ \frac{1}{\tilde{r}}=1$.

\end{theorem}

\begin{remark}
  We observe by hypotheses  $s>( \delta d/2\lambda ) +4 + d/p$, thus   by 
 by Sobolev embedding $(\varrho, \upsilon)   \in C\left([0, \tau_{m}(\omega)] ;  C_{b}^{4} \cap L^{1}(\R^{d})  \times   C_{b}^{4}(\R^{d}) \right)$. 
 Then for all $m$ there exist a constant  $C_{m}$ such that $\| \varrho \|_{C([0,\tau_{m}], C_{b}^{4}(\R^{d}) )}\leq C_{m} $ and $\| \varrho\|_{C([0,\tau_{m}], C_{b}^{4}(\R^{d}) )}|\leq C_{m} $, hence
 $\varrho_{t\wedge \tau_{m}} $ and  $\upsilon_{t\wedge \tau_{m}}$ are   $C^{4}$-semimartingales. Therefore,  we can apply Ito-Kunita-Wentzell formula to  $ \upsilon_{t\wedge \tau_{m}} (X_{t}^{k,N})$. 
 
\end{remark}

\begin{proof}

We denoted $V_{t,q}^{k,N},\upsilon_{q} $ the $q$-coordinate of $V_{t}^{k,N}$ and $\upsilon $
respectively. During the proof we use conveniently   the  Einstein's summation convention. 
In order not to load the notation during a large part of the proof, we use $t$ for $t\wedge \tau_{m}$. First we observe that 

\begin{align*}
Q_{t}^{N}=&   \frac{1}{N}\sum_{k=1}^{N}|V_{t}^{k,N}-\upsilon(X_{t}^{k,N},t)|^{2} +\left\| S_{t}^{N}\ast \phi_{N}^{r}-\varrho(., t)\right\|_{L^{2}(\R^{d})}^{2}\\
=&\dfrac{1}{N}\sum_{k=1}^{N}\vert V_{t}^{k,N}\vert^{2}-\dfrac{2}{N}\sum_{k=1}^{N} V_{t}^{k,N}\cdot\upsilon(X_{t}^{k,N},t)\\
&+\dfrac{1}{N}\sum_{k=1}^{N}\vert\upsilon(X_{t}^{k,N},t)\vert^{2}+\Vert  S_{t}^{N}\ast\phi_{N}^{r}\Vert_{L^{2}(\R^{d})}^{2}-2\langle   S_{t}^{N}\ast\phi_{N}^{r},\varrho(.,t)\rangle\\
&+\Vert\varrho(.,t)\Vert_{L^{2}(\R^{d})}^{2}\\
=&\dfrac{1}{N}\sum_{k=1}^{N}\vert V_{t}^{k,N}\vert^{2}-\dfrac{2}{N}\sum_{k=1}^{N} V_{t}^{k,N}\cdot\upsilon(X_{t}^{k,N},t)\\
&+\dfrac{1}{N}\sum_{k=1}^{N}\vert\upsilon(X_{t}^{k,N},t)\vert^{2}+\dfrac{1}{N^{2}}\sum_{k,l=1}^{N}\phi_{N}(X_{t}^{k,N}-X_{t}^{l,N})\\
&-\dfrac{2}{N}\sum_{k=1}^{N}\left(\varrho(.,t)\ast\phi_{N}^{r}\right)(X_{t}^{k,N})+\Vert\varrho(.,t)\Vert_{L^{2}(\R^{d})}^{2}.
\end{align*}

We will calculate  the  Stratonovich differential of each term in the sum of $Q_{t}^{N}$. By Ito formula we deduce

\begin{align*}
d\left(\dfrac{1}{N}\sum_{k=1}^{N}\vert V_{t}^{k,N}\vert^{2}\right)=&\dfrac{1}{N}\sum_{k=1}^{N}d\left(\vert V_{t}^{k,N}\vert^{2}\right)=\dfrac{1}{N}\sum_{k=1}^{N}2V_{t,q}^{k,N}\circ d(V_{t,q}^{k,N})\\
=&-\dfrac{2}{N}\sum_{k=1}^{N}V_{t,q}^{k,N}\nabla^{q}\left( S_{t}^{N}\ast\phi_{N}\right)(X_{t}^{k,N})dt\\
&-\dfrac{2}{N^{2}}\sum_{k=1}^{N}V_{t,q}^{k,N}\sum_{l=1 \atop l \neq k}^{N} \zeta_{N,qq^{\prime}}\left(X_{t}^{k,N}-X_{t}^{l,N}\right)\left(V_{t,q^{\prime}}^{k,N}-V_{t,q^{\prime}}^{l,N}\right)dt\\
&+\dfrac{2}{N}\sum_{k=1}^{N}V_{t,q}^{k,N}\sigma_{q}(X_{t}^{k,N})V_{t,q}^{k,N}    \circ d B_{t}^{q},
\end{align*}

where $q, q^{\prime}=1,..,d$ and we used  Einstein's summation convention. Applying Ito formula  for the product and
 Ito-Kunita-Wentzell formula  we  have 

\begin{align}\label{ito-wwent}
&d\left(-\dfrac{2}{N}\sum_{k=1}^{N} V_{t}^{k,N}\cdot\upsilon(X_{t}^{k,N},t)\right)=-\dfrac{2}{N}\sum_{k=1}^{N}d[ V_{t,q}^{k,N}\upsilon_{q}(X_{t}^{k,N},t)]\\
=&-\dfrac{2}{N}\sum_{k=1}^{N}\upsilon_{q}(X_{t}^{k,N},t)\circ d(V_{t,q}^{k,N})-\dfrac{2}{N}\sum_{k=1}^{N}V_{t,q}^{k,N}\circ d(\upsilon_{q}(X_{t}^{k,N},t))\nonumber\\
=&\dfrac{2}{N}\sum_{k=1}^{N}\upsilon(X_{t}^{k,N},t)\cdot\nabla\left( S_{t}^{N}\ast\phi_{N}\right)(X_{t}^{k,N})dt\nonumber\\
&+\dfrac{2}{N^{2}}\sum_{l,k=1}^{N}\upsilon(X_{t}^{k,N},t)\cdot \zeta_{N}\left(X_{t}^{k,N}-X_{t}^{l,N}\right)\left(V_{t}^{k,N}-V_{t}^{l,N}\right)dt\nonumber\\
&+\dfrac{2}{N}\sum_{k=1}^{N}V_{t,q}^{k,N}\upsilon(X_{t}^{k,N},t)\cdot\nabla\upsilon_{q}(X_{t}^{k,N},t)dt+\dfrac{2}{N}\sum_{k=1}^{N}V_{t}^{k,N}\cdot\nabla\varrho(X_{t}^{k,N},t)dt\nonumber\\
&-\dfrac{2}{N}\sum_{k=1}^{N}\frac{1}{2}V_{t}^{k,N}\cdot\nabla\left(\varrho^{2}\operatorname{div}_{x} \upsilon\right)(X_{t}^{k,N},t)dt\nonumber\\
&-\dfrac{2}{N}\sum_{k=1}^{N}V_{t,q}^{k,N}\frac{1}{2\varrho(X_{t}^{k,N},t)} \sum_{i=1}^{d} \nabla^{i}\left(\varrho^{2}\left[\nabla^{i} \upsilon_{q}+\nabla^{q} \upsilon_{i}\right]\right)(X_{t}^{k,N},t)dt\nonumber\\
&-\dfrac{2}{N}\sum_{k=1}^{N}V_{t,q}^{k,N}\nabla\upsilon_{q}(X_{t}^{k,N},t)\cdot V_{t}^{k,N}dt\nonumber\\
&-\dfrac{2}{N}\sum_{k=1}^{N}\upsilon_{q}(X_{t}^{k,N},t) \sigma_{q}(X_{t}^{k,N})V_{t,q}^{k,N} \circ dB_{t}^{q}\nonumber\\
&-\dfrac{2}{N}\sum_{k=1}^{N}V_{t,q}^{k,N}\sigma_{q}(X_{t}^{k,N})\upsilon_{q}(X_{t}^{k,N},t)  \circ dB_{t}^{q}\nonumber.
\end{align}

Applying Ito formula and Ito-Kunita-Wentzell formula  we obtain

\begin{align*}
&d\left(\dfrac{1}{N}\sum_{k=1}^{N}\vert\upsilon(X_{t}^{k,N},t)\vert^{2}\right)=\dfrac{2}{N}\sum_{k=1}^{N}\upsilon_{q}(X_{t}^{k,N},t)\circ d(\upsilon_{q}(X_{t}^{k,N},t))\\
=&-\dfrac{2}{N}\sum_{k=1}^{N}\upsilon_{q}(X_{t}^{k,N},t)\upsilon(X_{t}^{k,N},t)\cdot\nabla\upsilon_{q}(X_{t}^{k,N},t)dt-\dfrac{2}{N}\sum_{k=1}^{N}\upsilon(X_{t}^{k,N},t)\cdot\nabla\varrho(X_{t}^{k,N},t)dt\\
&+\dfrac{2}{N}\sum_{k=1}^{N}\frac{1}{2}\upsilon(X_{t}^{k,N},t)\cdot\nabla\left(\varrho^{2}\operatorname{div}_{x} \upsilon\right)(X_{t}^{k,N},t)dt\\
&+\dfrac{2}{N}\sum_{k=1}^{N}\upsilon_{q}(X_{t}^{k,N},t)\frac{1}{2\varrho(X_{t}^{k,N},t)} \sum_{i=1}^{d} \nabla^{i}\left(\varrho^{2}\left[\nabla^{i} \upsilon_{q}+\nabla^{q} \upsilon_{i}\right]\right)(X_{t}^{k,N},t)dt\\
&+\dfrac{2}{N}\sum_{k=1}^{N}\upsilon_{q}(X_{t}^{k,N},t)\nabla\upsilon_{q}(X_{t}^{k,N},t)\cdot V_{t}^{k,N}dt\\
&+\dfrac{2}{N}\sum_{k=1}^{N}\upsilon_{q}(X_{t}^{k,N},t)\sigma_{q}(X_{t}^{k,N})\upsilon_{q}(X_{t}^{k,N},t) \circ dB_{t}^{q}.
\end{align*}

By Leibniz rule and simple calculation we deduce

\begin{align*}
&d\left(\dfrac{1}{N^{2}}\sum_{k,l=1}^{N}\phi_{N}(X_{t}^{k,N}-X_{t}^{l,N})\right)=\dfrac{1}{N^{2}}\sum_{k,l=1}^{N}d\left(\phi_{N}(X_{t}^{k,N}-X_{t}^{l,N})\right)\\
&=\dfrac{1}{N^{2}}\sum_{k,l=1}^{N}\nabla^{q}\phi_{N}(X_{t}^{k,N}-X_{t}^{l,N}) \left(V_{t,q}^{k,N}-V_{t,q}^{l,N} \right)dt\\
&=\dfrac{1}{N}\sum_{k=1}^{N}\left(\dfrac{1}{N}\sum_{l=1}^{N}\nabla^{q}\phi_{N}(X_{t}^{k,N}-X_{t}^{l,N})\right)\ V_{t,q}^{k,N}dt\\
&\quad-\dfrac{1}{N}\sum_{l=1}^{N}\left(\dfrac{1}{N}\sum_{k=1}^{N}\nabla^{q}\phi_{N}(X_{t}^{k,N}-X_{t}^{l,N})\right) V_{t,q}^{l,N}dt \\
&=\dfrac{1}{N}\sum_{k=1}^{N}\left(\dfrac{1}{N}\sum_{l=1}^{N}\nabla^{q}\phi_{N}(X_{t}^{k,N}-X_{t}^{l,K})\right) V_{t,q}^{k,N}dt\\
&\quad-\dfrac{1}{N}\sum_{l=1}^{N}\left(-\dfrac{1}{N}\sum_{k=1}^{N}\nabla^{q}\phi_{N}(X_{t}^{l,N}-X_{t}^{k,N})\right) V_{t,q}^{l,N}dt \\
&=\dfrac{2}{N}\sum_{k=1}^{N}\left(\dfrac{1}{N}\sum_{l=1}^{N}\nabla\phi_{N}(X_{t}^{k,N}-X_{t}^{l,N})\right)\cdot V_{t}^{k,N}dt\\
&=\dfrac{2}{N}\sum_{k=1}^{N}V_{N}^{k}\cdot\nabla\left(S_{t}^{N}\ast\phi_{N}\right)(X_{t}^{k,N})dt.\\
\end{align*}

By Leibniz rule  we arrive at

\begin{align*}
&d\left(-\dfrac{2}{N}\sum_{k=1}^{N}\left(\varrho(.,t)\ast\phi_{N}^{r}\right)(X_{t}^{k,N})\right)
=-\dfrac{2}{N}\sum_{k=1}^{N}d\left(\left(\varrho(.,t)\ast\phi_{N}^{r}\right)(X_{t}^{k,N})\right)\\
&=-\dfrac{2}{N}\sum_{k=1}^{N}d\left(\int\varrho(x,t)\phi_{N}^{r}(x-X_{t}^{k,N})dx\right)\\
&=-\dfrac{2}{N}\sum_{k=1}^{N}V_{t}^{k,N}\cdot\nabla\left(\varrho(.,t)\ast\phi_{N}^{r}\right)(X_{t}^{k,N})dt-\dfrac{2}{N}\sum_{k=1}^{N}\left( d\varrho(.,t)\ast\phi_{N}^{r}\right)(X_{t}^{k,N})dt\\
&=-\dfrac{2}{N}\sum_{k=1}^{N}V_{t}^{k,N}\cdot\nabla\left(\varrho(.,t)\ast\phi_{N}^{r}\right)(X_{t}^{k,N})dt+\dfrac{2}{N}\sum_{k=1}^{N}\left(\operatorname{div}_{x}(\varrho(\cdot,t)\upsilon(\cdot,t))\ast\phi_{N}^{r}\right)(X_{t}^{k,N})dt,
\end{align*}

and 
\begin{align*}
d\left(\Vert\varrho(.,t)\Vert_{2}^{2}\right)=&-2\langle\varrho(\cdot,t),\operatorname{div}_{x}(\varrho(\cdot,t)\upsilon(\cdot,t))\rangle dt.
\end{align*}
From  lemma  \ref{Ito-correction}  we have 
\begin{align}\label{stochNs3}
& \dfrac{-4}{N}\sum_{k=1}^{N} \upsilon_{q}(X_{t}^{k,N},t) \sigma_{q}(X_{t}^{k,N}) V_{t,q}^{k,N}\circ dB_{t}^{q}\nonumber\\
&= \dfrac{-4}{N}\sum_{k=1}^{N} \upsilon_{q}(X_{t}^{k,N},t) \sigma_{q}(X_{t}^{k,N}) V_{t,q}^{k,N} dB_{t}^{q}-  \dfrac{4}{N}\sum_{k=1}^{N} \upsilon_{q}(X_{t}^{k,N},t) |\sigma_{q}(X_{t}^{k,N})|^{2} V_{t,q}^{k,N} dt,\nonumber\\
\end{align}

\begin{align}\label{stochNs1}
&\dfrac{2}{N}\sum_{k=1}^{N} V_{t,q}^{k,N}  \sigma_{q}(X_{t}^{k,N}) V_{t,q}^{k,N} \circ d B_{t}^{q}\nonumber\\
&=\dfrac{2}{N}\sum_{k=1}^{N} V_{t,q}^{k,N}  \sigma_{q}(X_{t}^{k,N}) V_{t,q}^{k,N}  d B_{t}^{q}+ \dfrac{2}{N}\sum_{k=1}^{N} |V_{t,q}^{k,N}|^{2}  |\sigma_{q}(X_{t}^{k,N})|^{2}  dt,
\end{align}
and 
\begin{align}\label{stochNs2}
&\dfrac{2}{N}\sum_{k=1}^{N}  \upsilon_{q}(X_{t}^{k,N},t) \sigma_{q}(X_{t}^{k,N}) \upsilon_{q}(X_{t}^{k,N},t)  \circ dB_{t}^{q}\nonumber\\
&=\dfrac{2}{N}\sum_{k=1}^{N} \upsilon_{q}(X_{t}^{k,N},t) \sigma_{q}(X_{t}^{k,N}) \upsilon_{q}(X_{t}^{k,N},t) dB_{t}^{q}+ \dfrac{2}{N}\sum_{k=1}^{N} |\upsilon_{q}(X_{t}^{k,N},t)|^{2} |\sigma_{q}(X_{t}^{k,N})|^{2} dt.
\end{align}

From  \eqref{stochNs3}, \eqref{stochNs1} and  \eqref{stochNs2} we obtain 

\begin{align*}
& dQ_{t}^{N}=-\dfrac{2}{N}\sum_{k=1}^{N}V_{t,q}^{k,N}\nabla^{q}\left( S_{t}^{N}\ast\phi_{N}\right)(X_{t}^{k,N})dt\\
&-\dfrac{2}{N^{2}}\sum_{k=1}^{N}V_{t,q}^{k,N}\sum_{l=1 \atop l \neq k}^{N} \zeta_{N,qq^{\prime}}\left(X_{t}^{k,N}-X_{t}^{l,N}\right)\left(V_{t,q^{\prime}}^{k,N}-V_{t,q^{\prime}}^{l,N}\right)dt\\
&+\dfrac{2}{N}\sum_{k=1}^{N}\upsilon(X_{t}^{k,N},t)\cdot\nabla\left( S_{t}^{N}\ast\phi_{N}\right)(X_{t}^{k,N})dt\\
&+\dfrac{2}{N^{2}}\sum_{l,k=1}^{N}\upsilon(X_{t}^{k,N},t)\cdot \zeta_{N}\left(X_{t}^{k,N}-X_{t}^{l,N}\right)\left(V_{t}^{k,N}-V_{t}^{l,N}\right)dt\\
&+\dfrac{2}{N}\sum_{k=1}^{N}V_{t,q}^{k,N}\upsilon(X_{t}^{k,N},t)\cdot\nabla\upsilon_{q}(X_{t}^{k,N},t)dt+\dfrac{2}{N}\sum_{k=1}^{N}V_{t}^{k,N}\cdot\nabla\varrho(X_{t}^{k,N},t)dt\\
&-\dfrac{2}{N}\sum_{k=1}^{N}\frac{1}{2}V_{t}^{k,N}\cdot\nabla\left(\varrho^{2}\operatorname{div}_{x} \upsilon\right)(X_{t}^{k,N},t)dt\\
&-\dfrac{2}{N}\sum_{k=1}^{N}V_{t,q}^{k,N}\frac{1}{2\varrho(X_{t}^{k,N},t)} \sum_{i=1}^{d} \nabla^{i}\left(\varrho^{2}\left[\nabla^{i} \upsilon_{q}+\nabla^{q} \upsilon_{i}\right]\right)(X_{t}^{k,N},t)dt\\
&-\dfrac{2}{N}\sum_{k=1}^{N}V_{t,q}^{k,N}\nabla\upsilon_{q}(X_{t}^{k,N},t)\cdot V_{t}^{k,N}dt-\dfrac{2}{N}\sum_{k=1}^{N}\upsilon_{q}(X_{t}^{k,N},t)\upsilon(X_{t}^{k,N},t)\cdot\nabla\upsilon_{q}(X_{t}^{k,N},t)dt\\
&-\dfrac{2}{N}\sum_{k=1}^{N}\upsilon(X_{t}^{k,N},t)\cdot\nabla\varrho(X_{t}^{k,N},t)dt+\dfrac{2}{N}\sum_{k=1}^{N}\frac{1}{2}\upsilon(X_{t}^{k,N},t)\cdot\nabla\left(\varrho^{2}\operatorname{div}_{x} \upsilon\right)(X_{t}^{k,N},t)dt\\
&+\dfrac{2}{N}\sum_{k=1}^{N}\upsilon_{q}(X_{t}^{k,N},t)\frac{1}{2\varrho(X_{t}^{k,N},t)} \sum_{i=1}^{d} \nabla^{i}\left(\varrho^{2}\left[\nabla^{i} \upsilon_{q}+\nabla^{q} \upsilon_{i}\right]\right)(X_{t}^{k,N},t)dt\\
&+\dfrac{2}{N}\sum_{k=1}^{N}\upsilon_{q}(X_{t}^{k,N},t)\nabla\upsilon_{q}(X_{t}^{k,N},t)\cdot V_{t}^{k,N}dt+\dfrac{2}{N}\sum_{k=1}^{N}V_{t}^{k,N}\cdot\nabla\left( S_{t}^{N}\ast\phi_{N}\right)(X_{t}^{k,N})dt\\
&-\dfrac{2}{N}\sum_{k=1}^{N}V_{t}^{k,N}\cdot\nabla\left(\varrho(.,t)\ast\phi_{N}^{r}\right)(X_{t}^{k,N})dt+\dfrac{2}{N}\sum_{k=1}^{N}\left(\operatorname{div}_{x}(\varrho(\cdot,t)\upsilon(\cdot,t))\ast\phi_{N}^{r}\right)(X_{t}^{k,N})dt\\
&-2\langle\varrho(\cdot,t),\operatorname{div}_{x}(
\varrho(\cdot,t)\upsilon(\cdot,t))\rangle dt+\dfrac{2}{N}\sum_{k=1}^{N} |\sigma_{q}(X_{t}^{k,N})|^{2}  |V_{t,q}^{k,N}-\upsilon_{q}(X_{t}^{k,N},t) |^{2}  dt\nonumber\\   \\
&+\dfrac{2}{N}\sum_{k=1}^{N} \sigma_{q}(X_{t}^{k,N})  |V_{t,q}^{k,N}-\upsilon_{q}(X_{t}^{k,N},t) |^{2}  dB_{t}^{q} \\
=&\sum_{j=1}^{19}D_{N}(j,t)dt  + D_{N}(20,t)dB_{t}^{q}.
\end{align*}

From now on we denote by $C>0$ a generic constant which may change line to line, and which may depend  on $m$. We observe that

\begin{eqnarray}\label{cotaAN0}
D_{N}(1,t)+D_{N}(15,t)=0 \nonumber\\
\end{eqnarray}

In exactly the same way as  in \cite{correa} and \cite{Oes}  we have

\begin{eqnarray}\label{cotaAN1}
A_{N}(1,t)&:=&D_{N}(3,t)+D_{N}(11,t)+D_{N}(17,t)+D_{N}(18,t) \nonumber\\
&\leq & C_{1}\left(\Vert S_{t}^{N}\ast\phi_{N}^{r}-\varrho(.,t)\Vert_{L^{2}(\R^{d})}^{2}+N^{-\beta/d}\right),  
\end{eqnarray}

\begin{eqnarray}\label{cotaAN2}
A_{N}(2,t)&:=&D_{N}(5,t)+D_{N}(9,t)+D_{N}(14,t)+D_{N}(10,t)\nonumber\\
&  \leq & C\frac{1}{N}\sum_{k=1}^{N}\Big|V_{t}^{k,N}-\upsilon(X_{t}^{k,N},t)\Big|^{2},
\end{eqnarray}

\begin{eqnarray}\label{cotaAN3}
A_{N}(3,t)&:=& D_{N}(6,t)+D_{N}(16,t)\nonumber\\
&\leqq&C_{3}\left(N^{-\beta/d}+\frac{1}{N}\sum_{k=1}^{N}\left|V_{t}^{k,N}-\upsilon(X_{t}^{k,N},t)\right|^{2}\right).
\end{eqnarray}

We get

\begin{eqnarray}\label{cotaAN4}
 D_{N}(19,t) & \leq &   C_{\sigma}  \dfrac{1}{N}\sum_{k=1}^{N} | V_{t}^{k,N}-\upsilon_{q}(X_{t}^{k,N},t)|^{2}.  
\end{eqnarray}

We observe 

\begin{align*}
& A_{4}:= D_{N}(2,t)+D_{N}(4,t)+D_{N}(7,t)+D_{N}(12,t)+D_{N}(8,t)+D_{N}(13,t)\nonumber\\
&=-\dfrac{2}{N^{2}}\sum_{k,l=1}^{N}\left(V_{t}^{k,N}-\upsilon(X_{t}^{k,N},t)\right)\cdot\zeta_{N}\left(X_{t}^{k,N}-X_{t}^{l,N}\right)\left(V_{t}^{k,N}-V_{t}^{l,N}\right)\nonumber\\
&-\dfrac{1}{N}\sum_{k=1}^{N}(V_{t}^{k,N}-\upsilon(X_{t}^{k,N},t))\cdot\nabla\left(\varrho^{2}\operatorname{div}_{x} \upsilon\right)(X_{t}^{k,N},t)\nonumber\\
&-\dfrac{1}{N}\sum_{k=1}^{N}(V_{t,q}^{k,N}-\upsilon_{q}(X_{t}^{k,N},t))\frac{1}{\varrho(X_{t}^{k,N},t)} \sum_{i=1}^{d} \nabla^{i}\left(\varrho^{2}\left[\nabla^{i} \upsilon_{q}+\nabla^{q} \upsilon_{i}\right]\right)(X_{t}^{k,N},t).
\end{align*}

Adding and subtracting the term
\begin{equation*}
 \frac{2}{N^{2}}\sum_{k,l=1}^{N}\left(V_{t}^{k,N}-\upsilon(X_{t}^{k,N},t)\right)\cdot\zeta_{N}(X_{t}^{k,N}-X_{t}^{l,N})\left(\upsilon(X_{t}^{k,N},t)-\upsilon(X_{t}^{l,N},t)\right)
 \end{equation*}
we get 
 \begin{align*}
&A_{N}(4,t)\\
=&-\dfrac{2}{N^{2}}\sum_{k,l=1}^{N}\left(V_{t}^{k,N}-\upsilon(X_{t}^{k,N},t)\right)\cdot\zeta_{N}\left(X_{t}^{k,N}-X_{t}^{l,N}\right)\left(V_{t}^{k,N}-V_{t}^{l,N}\right)\nonumber\\
&+\frac{2}{N^{2}}\sum_{k,l=1}^{N}\left(V_{t}^{k,N}-\upsilon(X_{t}^{k,N},t)\right)\cdot\zeta_{N}(X_{t}^{k,N}-X_{t}^{l,N})\left(\upsilon(X_{t}^{k,N},t)-\upsilon(X_{t}^{l,N},t)\right)\nonumber\\
&-\frac{2}{N^{2}}\sum_{k,l=1}^{N}\left(V_{t}^{k,N}-\upsilon(X_{t}^{k,N},t)\right)\cdot\zeta_{N}(X_{t}^{k,N}-X_{t}^{l,N})\left(\upsilon(X_{t}^{k,N},t)-\upsilon(X_{t}^{l,N},t)\right)\nonumber\\
&-\dfrac{1}{N}\sum_{k=1}^{N}(V_{t}^{k,N}-\upsilon(X_{t}^{k,N},t))\cdot\nabla\left(\varrho^{2}\operatorname{div}_{x} \upsilon\right)(X_{t}^{k,N},t)\nonumber\\
&-\dfrac{1}{N}\sum_{k=1}^{N}(V_{t,q}^{k,N}-\upsilon_{q}(X_{t}^{k,N},t))\frac{1}{\varrho(X_{t}^{k,N},t)} \sum_{i=1}^{d} \nabla^{i}\left(\varrho^{2}\left[\nabla^{i} \upsilon_{q}+\nabla^{q} \upsilon_{i}\right]\right)(X_{t}^{k,N},t).
\end{align*}
Since that 
\begin{align*}
&\frac{2}{N^{2}}\sum_{k,l=1}^{N}\left(V_{t}^{k,N}-\upsilon(X_{t}^{k,N},t)\right)\cdot\zeta_{N}(X_{t}^{k,N}-X_{t}^{l,N})\left(\upsilon(X_{t}^{k,N},t)-\upsilon(X_{t}^{l,N},t)\right)\\
&=\frac{2}{N}\sum_{k=1}^{N}\int_{\R^{d}}\left(V_{t}^{k,N}-\upsilon(X_{t}^{k,N},t)\right)\cdot\zeta_{N}(X_{t}^{k,N}-y)\left(\upsilon(X_{t}^{k,N},t)-\upsilon(y,t)\right)d( S_{t}^{N}(y)),
\end{align*}
we obtain
\begin{align}\label{202}
&A_{N}(4,t)=-\dfrac{2}{N^{2}}\sum_{k,l=1}^{N}\left(V_{t}^{k,N}-\upsilon(X_{t}^{k,N},t)\right)\cdot\zeta_{N}\left(X_{t}^{k,N}-X_{t}^{l,N}\right)\left(V_{t}^{k,N}-\upsilon(X_{t}^{k,N},t)\right.\nonumber\\
&\qquad\qquad\qquad\qquad\qquad\qquad\qquad\qquad\qquad\qquad\qquad\qquad\qquad\quad\left.-V_{t}^{l,N}+\upsilon(X_{t}^{l,N},t)\right)\nonumber\\
&-\frac{2}{N}\sum_{k=1}^{N}\int_{\R^{d}}\left(V_{t}^{k,N}-\upsilon(X_{t}^{k,N},t)\right)\cdot\zeta_{N}(X_{t}^{k,N}-y)\left(\upsilon(X_{t}^{k,N},t)-\upsilon(y,t)\right)d( S_{t}^{N}(y))\nonumber\\
&-\dfrac{1}{N}\sum_{k=1}^{N}(V_{t}^{k,N}-\upsilon(X_{t}^{k,N},t))\cdot\nabla\left(\varrho^{2}\operatorname{div}_{x} \upsilon\right)(X_{t}^{k,N},t)\nonumber\\
&-\dfrac{1}{N}\sum_{k=1}^{N}(V_{t,q}^{k,N}-\upsilon_{q}(X_{t}^{k,N},t))\frac{1}{\varrho(X_{t}^{k,N},t)} \sum_{i=1}^{d} \nabla^{i}\left(\varrho^{2}\left[\nabla^{i} \upsilon_{q}+\nabla^{q} \upsilon_{i}\right]\right)(X_{t}^{k,N},t).
\end{align}

By symmetry of the function  $\zeta_{N,qq^{\prime}}\left(\cdot\right)$ we deduce

\begin{align}\label{simetry}
&-\dfrac{1}{N^{2}}\sum_{k,l=1}^{N}\zeta_{N,qq^{\prime}}\left(X_{t}^{k,N}-X_{t}^{l,N}\right)\left(V_{t,q}^{k,N}-\upsilon_{q}(X_{t}^{k,N},t)\right)\left(V_{t,q^{\prime}}^{k,N}-\upsilon_{q^{\prime}}(X_{t}^{k,N},t)\right.\nonumber\\
&\quad\qquad\qquad\qquad\qquad\qquad\qquad\qquad\qquad\qquad\qquad\qquad\qquad\left.-V_{t,q^{\prime}}^{l,N}+\upsilon_{q^{\prime}}(X_{t}^{l,N},t)\right)\nonumber\\
&=-\dfrac{1}{N^{2}}\sum_{k,l=1}^{N}\zeta_{N,qq^{\prime}}\left(X_{t}^{l,N}-X_{t}^{k,N}\right)\left(\upsilon_{q}(X_{t}^{l,N},t)-V_{t,q}^{l,N}\right)\left(V_{t,q^{\prime}}^{k,N}-\upsilon_{q^{\prime}}(X_{t}^{k,N},t)\right.\nonumber\\
&\quad\qquad\qquad\qquad\qquad\qquad\qquad\qquad\qquad\qquad\qquad\qquad\qquad\left.-V_{t,q^{\prime}}^{l,N}+\upsilon_{q^{\prime}}(X_{t}^{l,N},t)\right)\nonumber\\
&=-\dfrac{1}{N^{2}}\sum_{k,l=1}^{N}\zeta_{N,qq^{\prime}}\left(X_{t}^{k,N}-X_{t}^{l,N}\right)\left(\upsilon_{q}(X_{t}^{l,N},t)-V_{t,q}^{l,N}\right)\left(V_{t,q^{\prime}}^{k,N}-\upsilon_{q^{\prime}}(X_{t}^{k,N},t)\right.\nonumber\\
&\quad\qquad\qquad\qquad\qquad\qquad\qquad\qquad\qquad\qquad\qquad\qquad\qquad\left.-V_{t,q^{\prime}}^{l,N}+\upsilon_{q^{\prime}}(X_{t}^{l,N},t)\right)
\end{align}
Using (\ref{simetry})  in the first term of  (\ref{202}) we have 

\begin{align}
&-\dfrac{2}{N^{2}}\sum_{k,l=1}^{N}\zeta_{N,qq^{\prime}}\left(X_{t}^{k,N}-X_{t}^{l,N}\right)\left(V_{t,q}^{k,N}-\upsilon_{q}(X_{t}^{k,N},t)\right)\left(V_{t,q^{\prime}}^{k,N}-\upsilon_{q^{\prime}}(X_{t}^{k,N},t)\right.\nonumber\\
&\quad\qquad\qquad\qquad\qquad\qquad\qquad\qquad\qquad\qquad\qquad\qquad\qquad\left.-V_{t,q^{\prime}}^{l,N}+\upsilon_{q^{\prime}}(X_{t}^{l,N},t)\right)\nonumber\\
&=-\dfrac{1}{N^{2}}\sum_{k,l=1}^{N}\zeta_{N,qq^{\prime}}\left(X_{t}^{k,N}-X_{t}^{l,N}\right)\left(V_{t,q}^{k,N}-\upsilon_{q}(X_{t}^{k,N},t)\right)\left(V_{t,q^{\prime}}^{k,N}-\upsilon_{q^{\prime}}(X_{t}^{k,N},t)\right.\nonumber\\
&\quad\qquad\qquad\qquad\qquad\qquad\qquad\qquad\qquad\qquad\qquad\qquad\qquad\left.-V_{t,q^{\prime}}^{l,N}+\upsilon_{q^{\prime}}(X_{t}^{l,N},t)\right)\nonumber\\
&-\dfrac{1}{N^{2}}\sum_{k,l=1}^{N}\zeta_{N,qq^{\prime}}\left(X_{t}^{k,N}-X_{t}^{l,N}\right)\left(V_{t,q}^{k,N}-\upsilon_{q}(X_{t}^{k,N},t)\right)\left(V_{t,q^{\prime}}^{k,N}-\upsilon_{q^{\prime}}(X_{t}^{k,N},t)\right.\nonumber\\
&\quad\qquad\qquad\qquad\qquad\qquad\qquad\qquad\qquad\qquad\qquad\qquad\qquad\left.-V_{t,q^{\prime}}^{l,N}+\upsilon_{q^{\prime}}(X_{t}^{l,N},t)\right)\nonumber\\
&=-\dfrac{1}{N^{2}}\sum_{k,l=1}^{N}\zeta_{N,qq^{\prime}}\left(X_{t}^{k,N}-X_{t}^{l,N}\right)\left(V_{t,q}^{k,N}-\upsilon_{q}(X_{t}^{k,N},t)\right)\left(V_{t,q^{\prime}}^{k,N}-\upsilon_{q^{\prime}}(X_{t}^{k,N},t)\right.\nonumber\\
&\quad\qquad\qquad\qquad\qquad\qquad\qquad\qquad\qquad\qquad\qquad\qquad\qquad\left.-V_{t,q^{\prime}}^{l,N}+\upsilon_{q^{\prime}}(X_{t}^{l,N},t)\right)\nonumber\\
&-\dfrac{1}{N^{2}}\sum_{k,l=1}^{N}\zeta_{N,qq^{\prime}}\left(X_{t}^{k,N}-X_{t}^{l,N}\right)\left(\upsilon_{q}(X_{t}^{l,N},t)-V_{t,q}^{l,N}\right)\left(V_{t,q^{\prime}}^{k,N}-\upsilon_{q^{\prime}}(X_{t}^{k,N},t)\right.\nonumber\\
&\quad\qquad\qquad\qquad\qquad\qquad\qquad\qquad\qquad\qquad\qquad\qquad\qquad\left.-V_{t,q^{\prime}}^{l,N}+\upsilon_{q^{\prime}}(X_{t}^{l,N},t)\right)\nonumber
\end{align}
\begin{align}\label{203}
&=-\dfrac{1}{N^{2}}\sum_{k,l=1}^{N}\zeta_{N,qq^{\prime}}\left(X_{t}^{k,N}-X_{t}^{l,N}\right)\left(V_{t,q}^{k,N}-\upsilon_{q}(X_{t}^{k,N},t)-V_{t,q}^{l,N}+\upsilon(X_{t}^{l,N},t)\right)\nonumber\\
&\qquad\qquad\qquad\qquad\qquad\qquad\qquad\qquad\left(V_{t,q^{\prime}}^{k,N}-\upsilon_{q^{\prime}}(X_{t}^{k,N},t)-V_{t,q^{\prime}}^{l,N}+\upsilon_{q^{\prime}}(X_{t}^{l,N},t)\right).
\end{align}

From (\ref{202}) and (\ref{203}) we get

\begin{align}\label{204}
&A_{N}(4,t)\nonumber\\
=&-\dfrac{1}{N^{2}}\sum_{k,l=1}^{N}\zeta_{N,qq^{\prime}}\left(X_{t}^{k,N}-X_{t}^{l,N}\right)\left(V_{t,q}^{k,N}-\upsilon_{q}(X_{t}^{k,N},t)-V_{t,q}^{l,N}+\upsilon(X_{t}^{l,N},t)\right)\nonumber\\
&\qquad\qquad\qquad\qquad\qquad\qquad\qquad\qquad\left(V_{t,q^{\prime}}^{k,N}-\upsilon_{q^{\prime}}(X_{t}^{k,N},t)-V_{t,q^{\prime}}^{l,N}+\upsilon_{q^{\prime}}(X_{t}^{l,N},t)\right)\nonumber\\
&-\frac{2}{N}\sum_{k=1}^{N}\int_{\R^{d}}\left(V_{t}^{k,N}-\upsilon(X_{t}^{k,N},t)\right)\cdot\zeta_{N}(X_{t}^{k,N}-y)\left(\upsilon(X_{t}^{k,N},t)-\upsilon(y,t)\right)d( S_{t}^{N}(y))\nonumber\\
&-\dfrac{1}{N}\sum_{k=1}^{N}(V_{t}^{k,N}-\upsilon(X_{t}^{k,N},t))\cdot\nabla\left(\varrho^{2}\operatorname{div}_{x} \upsilon\right)(X_{t}^{k,N},t)\nonumber\\
&-\dfrac{1}{N}\sum_{k=1}^{N}(V_{t,q}^{k,N}-\upsilon_{q}(X_{t}^{k,N},t))\frac{1}{\varrho(X_{t}^{k,N},t)} \sum_{i=1}^{d} \nabla^{i}\left(\varrho^{2}\left[\nabla^{i} \upsilon_{q}+\nabla^{q} \upsilon_{i}\right]\right)(X_{t}^{k,N},t).
\end{align}

Adding and subtracting 

\begin{equation*}
\dfrac{2}{N}\sum_{k=1}^{N}\int_{\mathbb{R}^{d}}\varrho(y, t)\left(V_{t}^{k,N}-\upsilon(X_{t}^{k,N},t)\right)\cdot\zeta_{N}(X_{t}^{k,N}-y)\left(\upsilon(X_{t}^{k,N}, t)-\upsilon(y, t)\right)dy    
\end{equation*}
in  \eqref{204}, we obtain 

\begin{eqnarray}
A_{N}(4,t)=R_{N,1}+R_{N,2}+R_{N,3},
\end{eqnarray}
where
\begin{eqnarray*}
 R_{N,1} &\\
:=&-\dfrac{1}{N^{2}}\sum_{k,l=1}^{N}\left(V_{t}^{k,N}-\upsilon(X_{t}^{k,N},t)-V_{t}^{l,N}+\upsilon(X_{t}^{l,N},t)\right)\cdot\zeta_{N}\left(X_{t}^{k,N}-X_{t}^{l,N}\right)\nonumber\\
&&\qquad\qquad\qquad\qquad\qquad\left(V_{t}^{k,N}-\upsilon(X_{t}^{k,N},t)-V_{t}^{l,N}+\upsilon(X_{t}^{l,N},t)\right),
\end{eqnarray*}

\begin{align*}
  &R_{N,2}\\
  :=&\dfrac{2}{N}\sum_{k=1}^{N}\int_{\mathbb{R}^{d}}\varrho(y, t)\left(V_{t}^{k,N}-\upsilon(X_{t}^{k,N},t)\right)\cdot\zeta_{N}(X_{t}^{k,N}-y)\\
&\quad\qquad\qquad\qquad\qquad\qquad\qquad\qquad\qquad\qquad\qquad\left(\upsilon(X_{t}^{k,N}, t)-\upsilon(y, t)\right)dy\nonumber\\
&-\frac{2}{N}\sum_{k=1}^{N}\int_{\R^{d}}\left(V_{t}^{k,N}-\upsilon(X_{t}^{k,N},t)\right)\cdot\zeta_{N}(X_{t}^{k,N}-y)\left(\upsilon(X_{t}^{k,N},t)-\upsilon(y,t)\right)d( S_{t}^{N}(y))\nonumber\\  
\end{align*}
and 

\begin{align*}
    &R_{N,3}\\
    :=&-\dfrac{1}{N}\sum_{k=1}^{N}(V_{t}^{k,N}-\upsilon(X_{t}^{k,N},t))\cdot\nabla\left(\varrho^{2}\operatorname{div}_{x} \upsilon\right)(X_{t}^{k,N},t)\nonumber\\
&-\dfrac{1}{N}\sum_{k=1}^{N}(V_{t,q}^{k,N}-\upsilon_{q}(X_{t}^{k,N},t))\frac{1}{\varrho(X_{t}^{k,N},t)} \sum_{i=1}^{d} \nabla^{i}\left(\varrho^{2}\left[\nabla^{i} \upsilon_{q}+\nabla^{q} \upsilon_{i}\right]\right)(X_{t}^{k,N},t)\nonumber\\
&-\dfrac{2}{N}\sum_{k=1}^{N}\int_{\mathbb{R}^{d}}\varrho(y, t)\left(V_{t}^{k,N}-\upsilon(X_{t}^{k,N},t)\right)\cdot\zeta_{N}(X_{t}^{k,N}-y)\left(\upsilon(X_{t}^{k,N}, t)-\upsilon(y, t)\right)dy.
\end{align*}

Now, we will estimated the terms  $R_{N,1}$, $R_{N,2}$ and  $R_{N,3}$. From lemma \ref{lemmaR1} we have

\begin{eqnarray}\label{estimativaRN1}
R_{N, 1}&\leq& C (Q_{t}^{N} +  (Q_{t}^{N})^{3/2} N^{\gamma(d+4) / 2 d}).
\end{eqnarray}

From lemma \ref{lemmaR2} obtain

\begin{eqnarray}\label{estimativaRN2}
&&\left|R_{N, 2}\right| \leq
 C\left(Q_{t}^{N}+N^{-2 \beta / d} N^{2 \gamma(d+2) / d}+(Q_{t}^{N})^{3 / 2} N^{\gamma(d+4) / 2 d}\right).
\end{eqnarray}

Thus from  \eqref{estimativaRN1} and \eqref{estimativaRN2}, we  have  that there exist   $C>0$ such that 
\begin{eqnarray}\label{estiR1R2}
    R_{N,1}+R_{N,2}&\leq&  C\left(Q_{t}^{N}+N^{-2 \beta / d} N^{2 \gamma(d+2) / d}+(Q_{t}^{N})^{3 / 2} N^{\gamma(d+4) / 2 d}\right).
\end{eqnarray}

From lemma \ref{lemmaR3} we deduce 

\begin{eqnarray}\label{estiRN3}
        |R_{N,3}|&\leq  C\left(N^{-\gamma/d}+\frac{1}{N}\sum_{k=1}^{N}|V_{t}^{k,N}-\upsilon(X_{t}^{k,N},t)|^{2}\right).
    \end{eqnarray}

From  \eqref{estiR1R2} e \eqref{estiRN3}, there exist  $C>0$ such that 
\begin{equation}\label{EsQ6}
  A_{N}(4,t)\leq C  \left(Q_{t}^{N}+N^{-2 \beta / d} N^{2 \gamma(d+2) / d}+(Q_{t}^{N})^{3 / 2} N^{\gamma(d+4) / 2 d}+N^{-\gamma / d}\right).
\end{equation}

From \eqref{cotaAN0}-\eqref{cotaAN4} and  \eqref{EsQ6}, we conclude 
\begin{align*}
   Q_{t\wedge\tau_{m}}^{N}
\leq &Q_{0}^{N}+C\int_{0}^{t}\left(Q_{s\wedge\tau_{m}}^{N}+N^{-\beta / d}+N^{-\gamma/ d}+N^{-2 \beta / d} N^{2 \gamma(d+2) / d}\right.\\
&\quad\qquad\qquad\qquad\qquad\qquad\qquad\qquad\left.+N^{\gamma(d+4) / 2 d}(Q_{s\wedge\tau_{m}}^{N})^{3 / 2} \right)ds\\
&+\left|\int_0^{t \wedge \tau_m} \frac{2}{N} \sum_{k=1}^N \sigma_q\left(X_s^{k, N}\right)\left| V_{s, q}^{k, N}-\upsilon_q\left(X_s^{k, N}, s\right)\right|^2 d B_s^q \right|,
\end{align*}

com $0\leq t\leq T$. We set 
\begin{equation*}
    T_{N}=\inf \left\{t \geq 0:  Q_{t}^{N} \geq N^{-\delta}\right\} \wedge T . 
\end{equation*}

Then we have 

\begin{align*}
Q_{t \wedge T_N \wedge \tau_m}^N \leq& Q_0^N  +C \int_0^t\left(Q_{s \wedge T_N \wedge \tau_m}^N+N^{-\beta / d}+N^{-\gamma / d}+N^{-2 \beta / d} N^{2 \gamma(d+2) / d}\right.\nonumber\\
&\left.\qquad\qquad\qquad\qquad+N^{\gamma(d+4) / 2 d}\left(Q_{s \wedge T_N \wedge \tau_m}\right)^{3 / 2}\right) d s \\
& +\left|\int_0^{t \wedge T_N \wedge \tau_m} \frac{2}{N} \sum_{k=1}^N \sigma_q\left(X_s^{k, N}\right)\left| V_{s, q}^{k, N}-\upsilon_q\left(X_s^{k, N}, s\right)\right|^2 d B_s^q \right|.
\end{align*}

Taking supreme,   we obtain 

\begin{align*}
\sup_{[0,t]}Q_{t \wedge T_N \wedge \tau_m}^N \leq& Q_0^N  +C \int_0^t\left(\sup_{[0,s]}Q_{u \wedge T_N \wedge \tau_m}^N+N^{-\beta / d}+N^{-\gamma / d}\right.\nonumber\\
&\left.\qquad+N^{-2 \beta / d} N^{2 \gamma(d+2) / d}+N^{\gamma(d+4) / 2 d}\left(\sup_{[0,s]}Q_{u \wedge T_N \wedge \tau_m}\right)^{3 / 2}\right) d s \\
& +\sup_{[0,t]}\left|\int_0^{t \wedge T_N \wedge \tau_m} \frac{2}{N} \sum_{k=1}^N \sigma_q\left(X_s^{k, N}\right)\left| V_{s, q}^{k, N}-\upsilon_q\left(X_s^{k, N}, s\right)\right|^2 d B_s^q \right|.
\end{align*}

By Holder inequality  and taking expectation  we deduce 
\begin{align*}
&\EE\left(\sup _{[0, t]} Q_{s \wedge T_N \wedge \tau_m}^N\right)^2\\
\leq& \EE\left(Q_0^N\right)^2+C_T \int_0^t\left(\EE\left(\sup _{[0, s]} Q_{u \wedge T_N \wedge \tau_m}^N\right)^2+N^{-2 \beta / d}+N^{-2 \gamma / d}\right. \\
&\qquad\qquad\qquad\qquad\left.+N^{-4 \beta / d} N^{4 \gamma(d+2) / d}+N^{\gamma(d+4) / d} \EE\ \left(\sup _{[0, u]} Q_{u \wedge T_N \wedge \tau_m}^N\right)^3\right) d s \\
&+\EE\left(\sup_{[0,t]}\left|\int_0^{t \wedge T_N \wedge \tau_m} \frac{2}{N} \sum_{k=1}^N \sigma_q\left(X_s^{k, N}\right)\left| V_{s, q}^{k, N}-\upsilon_q\left(X_s^{k, N}, s\right)\right|^2 d B_s^q \right|\right)^{2}.
\end{align*}
By definition of $T_{N}$  and maximal martingale inequality we have 
\begin{align*}
&\EE\left(\sup _{[0, t]} Q_{s \wedge T_N \wedge \tau_m}^N\right)^2\\
\leq& \EE\left(Q_0^N\right)^2+C_T \int_0^t\left(\EE\left(\sup _{[0, s]} Q_{u \wedge T_N \wedge \tau_m}^N\right)^2+N^{-2 \beta / d}+N^{-2 \gamma / d}\right. \\
&\qquad\qquad\qquad\qquad\left.+N^{-4 \beta / d} N^{4 \gamma(d+2) / d}+N^{\gamma(d+4) / d} N^{-3 \delta}\right) d s \\
&+\EE\left|\int_0^{t \wedge T_N \wedge \tau_m} \frac{2}{N} \sum_{k=1}^N \sigma_q\left(X_s^{k, N}\right)\left| V_{s, q}^{k, N}-\upsilon_q\left(X_s^{k, N}, s\right)\right|^2 d B_s^q \right|^2.
\end{align*}

By Ito isometry , by boundedness of $\sigma$ and  taking supreme we have 
\begin{align*}
 &\EE\left|\int_0^{t \wedge T_N \wedge \tau_m} \frac{2}{N} \sum_{k=1}^N \sigma_q\left(X_s^{k, N}\right)\left| V_{s, q}^{k, N}-\upsilon_q\left(X_s^{k, N}, s\right)\right|^2 d B_s^q \right|^2\\
 =&\EE  \int_0^{t \wedge T_N \wedge \tau_m}  \big| \frac{2}{N}  |\sum_{k=1}^N \sigma_q\left(X_s^{k, N}\right)\left| V_{s, q}^{k, N}-\upsilon_q\left(X_s^{k, N}, s\right)\right|^2 \big|^{2} ds\\
 \leq& C\EE  \int_0^{t \wedge T_N \wedge \tau_m}  \big| \frac{2}{N} \sum_{k=1}^N \left| V_{s, q}^{k, N}-\upsilon_q\left(X_s^{k, N}, s\right)\right|^{2} \big|^{2} ds\\
 \leq& C\EE  \int_0^{t \wedge T_N \wedge \tau_m}  (Q_{s}^{N})^{2} ds \\
 \leq& C \int_0^{t}  \EE\left(\sup_{[0, s]} Q_{u \wedge T_N \wedge \tau_m}^N\right)^2 ds.
\end{align*}
Therefore we deduce 
\begin{align*}
&\EE\left(\sup _{[0, t]} Q_{s \wedge T_N \wedge \tau_m}^N\right)^2\\
\leq& \EE\left(Q_0^N\right)^2+C_T \int_0^t\left(\EE\left(\sup _{[0, s]} Q_{u \wedge T_N \wedge \tau_m}^N\right)^2+N^{-2 \beta / d}+N^{-2 \gamma / d}\right.\nonumber\\
&\left.\qquad\qquad\qquad\qquad+N^{-4 \beta / d} N^{4 \gamma(d+2) / d}+N^{\gamma(d+4) / d} N^{-3 \delta}\right) d s \\
&+C \int_0^t \EE\left(\sup_{[0, s]} Q_{u \wedge T_N \wedge \tau_m}^N\right)^2 d s.
\end{align*}

 By Gronwal lemma we conclude 
\begin{align*}
  \EE\left(\sup _{[0, t]} Q_{s \wedge T_N \wedge \tau_m}^N\right)^2 \leq& C_T \EE\left(Q_0^N\right)^2+C_T N^{-2 \beta / d}+N^{-2 \gamma / d}+N^{-4 \beta / d} N^{4 \gamma(d+2) / d}\\
  &+N^{\gamma(d+4) / d} N^{-3 \delta}. 
\end{align*}
Therefore from  \eqref{paramdelta} and  (\ref{condi,ini.Q0N}) 
\begin{equation}\label{dmEQtN}
  \lim _{N \rightarrow \infty} N^{2 \delta} \EE\left(\sup _{\left[0, T_N\right]} Q_{s \wedge \tau_m}^N\right)^2=0.  
\end{equation}

Suppose there are infinite numbers $N_{k}$ such that $T^{N_{k}}< T$. From  (\ref{dmEQtN})    we have that there exist  a sub-sequence $N_{j}$ of $N_{k}$   such that 
almost surely we have 

\[
\lim_{j\rightarrow\infty} N_{j}^{\delta} \sup_{[0,T_{N_{j}}]} Q_{s\wedge  \tau_{m}  }^{N_{j}} =0. 
\]

Therefore  for $j$ large we obtain 

$$
N_{j}^{\delta} Q_{T_{N_{j}} \wedge  \tau_{m}  }^{N_{j}}<  \epsilon,
$$

this  contradicts the definition of $T_{N_{j}}$.  So we conclude that $T^{N}=T$
 for $N$ large enough.  Thus we have

\begin{eqnarray}\label{convefinal}
\lim _{N \rightarrow \infty}  N^{2\delta} \mathbb{E} \big( \sup_{[0,T]} Q_{s\wedge  \tau_{m}  }^{N}  \big) ^{2}  =0.
\end{eqnarray}

Finally we will show the convergence  of  $S_{t}^{N}$ e $V_{t}^{N}$ to  $\upsilon$ e $\varrho \upsilon$  respectively.
 
\noindent Let $f\in  E_{2,r}^{\alpha}$, by  \eqref{cotafNbeta}  we deduce

\begin{eqnarray}\label{limXNvarrNS}
|\langle S_{t}^{N},f\rangle-\langle\varrho(.,t),f \rangle|&=&|\langle S_{t}^{N},f-f\ast\phi_{N}^{r}+f\ast\phi_{N}^{r}\rangle-\langle\varrho(.,t),f \rangle|\nonumber\\
&=&|\langle S_{t}^{N},f-f\ast\phi_{N}^{r}\rangle+\langle S_{t}^{N}\ast\phi_{N}^{r},f\rangle-\langle\varrho(.,t),f \rangle|\nonumber\\
&=&|\langle S_{t}^{N},f-f\ast\phi_{N}^{r}\rangle+\langle S_{t}^{N}\ast\phi_{N}^{r}-\varrho(.,t),f \rangle|\nonumber\\
&\leqq&|\langle S_{t}^{N},f-f\ast\phi_{N}^{r}\rangle|+|\langle S_{t}^{N}\ast\phi_{N}^{r}-\varrho(.,t),f \rangle|\nonumber\\
&\leqq&|\langle S_{t}^{N},1\rangle|\left\Vert f-f\ast\phi_{N}^{r}\right\Vert_{\infty}+|\langle S_{t}^{N}\ast\phi_{N}^{r}-\varrho(,t),f \rangle|\nonumber\\
&\leqq&\left\Vert f-f\ast\phi_{N}^{r}\right\Vert_{\infty}+\left\Vert f\right\Vert_{L^{2}(\R^{d})}\left\Vert S_{t}^{N}\ast\phi_{N}^{r}-\varrho(.,t)\right\Vert_{L^{2}(\R^{d})}\nonumber\\
&\leqq&CN^{-\beta/d}\left\Vert\nabla f\right\Vert_{\infty}+\left\Vert f\right\Vert_{L^{2}(\R^{d})}\left\Vert S_{t}^{N}\ast\phi_{N}^{r}-\varrho(.,t)\right\Vert_{L^{2}(\R^{d})}.\nonumber\\
\end{eqnarray}

From \eqref{limXNvarrNS}  and  \eqref{Sobolev} we have

\[
\lim _{N \rightarrow \infty}  N^{2\delta} \mathbb{E} \sup_{[0, T]} \|S_{t\wedge  \tau_{m} }^{N}- \varrho_{t \wedge  \tau_{m} } \|_{E_{2,\tilde{r}}^{-\alpha}}^{2}
=0. 
\]

\noindent Let  $g\in  E_{2,r}^{\alpha}$ we obtain 
\begin{eqnarray}\label{convV}
&&\left|\left\langle V_{t}^{N}-\varrho(.,t)\upsilon(.,t),g\right\rangle\right|\nonumber\\
&&\quad=\left|\frac{1}{N}\sum_{k=1}^{N}V_{t}^{k,N}\cdot g(X_{t}^{k,N})-\left\langle\varrho(.,t),\upsilon(.,t)\cdot g\right\rangle\right|\nonumber\\
&&\quad=\left|\frac{1}{N}\sum_{k=1}^{N}(V_{t}^{k,N}-\upsilon(X_{t}^{k,N}),t)\cdot g(X_{t}^{k,N})+\left\langle S_{t}^{N}-\varrho(.,t),\upsilon(.,t)\cdot g\right\rangle\right|\nonumber\\
&&\quad\leqq\frac{1}{N}\sum_{k=1}^{N}\left|V_{t}^{k,N}-\upsilon(X_{t}^{k,N},t)\right|\left| g(X_{t}^{k,N})\right|\nonumber+\left|\left\langle S_{t}^{N}-\varrho(.,t),\upsilon(.,t)\cdot g\right\rangle\right|\nonumber\\
&&\quad\leqq\Vert g\Vert_{\infty}\frac{1}{N}\sum_{k=1}^{N}\left|V_{t}^{k,N}-\upsilon(X_{t}^{k,N},t)\right|+|\langle S_{t}^{N}-\varrho(.,t),\upsilon(.,t)\cdot g\rangle|\nonumber\\
&&\quad\leqq C\Vert g\Vert_{\infty}\left(\frac{1}{N}\sum_{k=1}^{N}|V_{t}^{k,N}-\upsilon(X_{t}^{k,N},t)|^{2} \right)^{1/2}+|\langle S_{t}^{N}-\varrho(.,t),\upsilon(.,t)\cdot g\rangle|\nonumber\\
&&
\end{eqnarray}

From  \eqref{convV} and   \eqref{Sobolev}  we get

\[
\lim _{N \rightarrow \infty}  N^{2\delta} \mathbb{E} \sup_{[0, T]} \|V_{t\wedge  \tau_{m} }^{N}- (\varrho \upsilon)_{t \wedge  \tau_{m} } \|_{E_{2,\tilde{r}}^{-\alpha}}^{2}.
=0. 
\]

\end{proof}

\section{Apendix}

\begin{lemma}\label{Ito-correction} We have 
  \begin{align}\label{cov1}
&\dfrac{-4}{N}\sum_{k=1}^{N} \upsilon_{q}(X_{t}^{k,N},t) \sigma_{q}(X_{t}^{k,N}) V_{t,q}^{k,N}\circ dB_{t}^{q}\nonumber\\
&= \dfrac{-4}{N}\sum_{k=1}^{N} \upsilon_{q}(X_{t}^{k,N},t) \sigma_{q}(X_{t}^{k,N}) V_{t,q}^{k,N} dB_{t}^{q}-  \dfrac{4}{N}\sum_{k=1}^{N} \upsilon_{q}(X_{t}^{k,N},t) |\sigma_{q}(X_{t}^{k,N})|^{2} V_{t,q}^{k,N} dt,\nonumber\\
\end{align}
\begin{align}\label{cov2}
&\dfrac{2}{N}\sum_{k=1}^{N} V_{t,q}^{k,N}  \sigma_{q}(X_{t}^{k,N}) V_{t,q}^{k,N} \circ d B_{t}^{q}\nonumber\\
&=\dfrac{2}{N}\sum_{k=1}^{N} V_{t,q}^{k,N}  \sigma_{q}(X_{t}^{k,N}) V_{t,q}^{k,N}  d B_{t}^{q}+ \dfrac{2}{N}\sum_{k=1}^{N} |V_{t,q}^{k,N}|^{2}  |\sigma_{q}(X_{t}^{k,N})|^{2}  dt,
\end{align}
and 
\begin{align}\label{cov3}
&\dfrac{2}{N}\sum_{k=1}^{N}  \upsilon_{q}(X_{t}^{k,N},t) \sigma_{q}(X_{t}^{k,N}) \upsilon_{q}(X_{t}^{k,N},t)  \circ dB_{t}^{q}\nonumber\\
&=\dfrac{2}{N}\sum_{k=1}^{N} \upsilon_{q}(X_{t}^{k,N},t) \sigma_{q}(X_{t}^{k,N}) \upsilon_{q}(X_{t}^{k,N},t) dB_{t}^{q}+ \dfrac{2}{N}\sum_{k=1}^{N} |\upsilon_{q}(X_{t}^{k,N},t)|^{2} |\sigma_{q}(X_{t}^{k,N})|^{2} dt. 
\end{align}

\end{lemma}

\begin{proof} We will verifies (\ref{cov1}). 
The relation between Ito and Stratonovich integrals is (see \cite{Ku2})

\[
\dfrac{2}{N}\sum_{k=1}^{N} \upsilon_{q}(X_{t}^{k,N},t) \sigma_{q}(X_{t}^{k,N}) V_{t,q}^{k,N}\circ dB_{t}^{q}
\]

\[
=\dfrac{2}{N}\sum_{k=1}^{N} \upsilon_{q}(X_{t}^{k,N},t) \sigma_{q}(X_{t}^{k,N}) V_{t,q}^{k,N}  dB_{t}^{q} + 
\frac{1}{2}  \left[  \dfrac{2}{N}\sum_{k=1}^{N} \upsilon_{q}(X_{.}^{k,N},.) \sigma_{q}(X_{.}^{k,N}) V_{.,q}^{k,N}, B_{.}  \right]_{t}
\]

where $ \left[.,. \right]_{t} $ denotes the joint quadratic variation. Applying Ito  formula for the product,  Ito-Kunita-Wentzell formula   and  considering the  expression  of   $d[V_{t,q}^{k,N}\upsilon_{q}(X_{t}^{k,N},t)]$ given in  \eqref{ito-wwent}, we 
have 

\begin{align*}
&d\left(-\dfrac{4}{N}\sum_{k=1}^{N} \upsilon_{q}(X_{t}^{k,N},t) \sigma_{q}(X_{t}^{k,N}) V_{t,q}^{k,N}\right)\nonumber\\
=&-\dfrac{4}{N}\sum_{k=1}^{N}\sigma_{q}(X_{t}^{k,N})\circ d[V_{t,q}^{k,N}\upsilon_{q}(X_{t}^{k,N},t)]-\dfrac{4}{N}\sum_{k=1}^{N} \upsilon_{q}(X_{t}^{k,N},t) V_{t,q}^{k,N}\circ d(\sigma_{q}(X_{t}^{k,N}))\nonumber\\
=&\dfrac{4}{N}\sum_{k=1}^{N}\sigma_{q}(X_{t}^{k,N})\upsilon_{q}(X_{t}^{k,N},t)\nabla^{q}\left( S_{t}^{N}\ast\phi_{N}\right)(X_{t}^{k,N})dt\nonumber\\
&+\dfrac{2}{N^{2}}\sum_{l,k=1}^{N}\sigma_{q}(X_{t}^{k,N})\upsilon_{q}(X_{t}^{k,N},t)\zeta_{N,q,\ast}\left(X_{t}^{k,N}-X_{t}^{l,N}\right)\left(V_{t}^{k,N}-V_{t}^{l,N}\right)dt\nonumber\\
&+\dfrac{4}{N}\sum_{k=1}^{N}\sigma_{q}(X_{t}^{k,N})V_{t,q}^{k,N}\upsilon(X_{t}^{k,N},t)\cdot\nabla\upsilon_{q}(X_{t}^{k,N},t)dt\nonumber\\
&+\dfrac{4}{N}\sum_{k=1}^{N}\sigma_{q}(X_{t}^{k,N})V_{t,q}^{k,N}\nabla^{q}\varrho(X_{t}^{k,N},t)dt\nonumber\\
&-\dfrac{4}{N}\sum_{k=1}^{N}\frac{1}{2}\sigma_{q}(X_{t}^{k,N})V_{t,q}^{k,N}\nabla^{q}\left(\varrho^{2}\operatorname{div}_{x} \upsilon\right)(X_{t}^{k,N},t)dt\nonumber\\
&-\dfrac{4}{N}\sum_{k=1}^{N}\sigma_{q}(X_{t}^{k,N})V_{t,q}^{k,N}\frac{1}{2\varrho(X_{t}^{k,N},t)} \sum_{i=1}^{d} \nabla^{i}\left(\varrho^{2}\left[\nabla^{i} \upsilon_{q}+\nabla^{q} \upsilon_{i}\right]\right)(X_{t}^{k,N},t)dt
\end{align*}
\begin{align}\label{upsiqVqrhoqN}
&-\dfrac{4}{N}\sum_{k=1}^{N}\sigma_{q}(X_{t}^{k,N})V_{t,q}^{k,N}\nabla\upsilon_{q}(X_{t}^{k,N},t)\cdot V_{t}^{k,N}dt\quad\qquad\qquad\qquad\qquad\qquad\qquad\nonumber\\
&-\dfrac{8}{N}\sum_{k=1}^{N}V_{t,q}^{k,N}|\sigma_{q}(X_{t}^{k,N})|^{2}\upsilon_{q}(X_{t}^{k,N},t)  \circ dB_{t}^{q}\nonumber\\
&-\dfrac{4}{N}\sum_{k=1}^{N} \upsilon_{q}(X_{t}^{k,N},t) V_{t,q}^{k,N}\nabla\sigma_{q}(X_{t}^{k,N})\cdot V_{t}^{k,N}dt.
\end{align}
Only the martingale part of  $-\dfrac{4}{N}\sum_{k=1}^{N} \upsilon_{q}(X_{t}^{k,N},t) \sigma_{q}(X_{t}^{k,N}) V_{t,q}^{k,N}$
counts in the joint quadratic variation.   From  \eqref{upsiqVqrhoqN}  the only term contributing to the covariance is
\begin{align*}
   & \left[-\dfrac{4}{N}\sum_{k=1}^{N} \upsilon_{q}(X_{t}^{k,N},t) \sigma_{q}(X_{t}^{k,N}) V_{t,q}^{k,N},B_{t}^{q}\right]\\
   =&\left[-\dfrac{8}{N}\sum_{k=1}^{N}\int_{0}^{t}V_{s,q}^{k,N}|\sigma_{q}(X_{s}^{k,N})|^{2}\upsilon_{q}(X_{s}^{k,N},s)  \circ dB_{s}^{q},B_{t}^{q}\right]\\
    =&-\dfrac{8}{N}\sum_{k=1}^{N}\int_{0}^{t}V_{s,q}^{k,N}|\sigma_{q}(X_{s}^{k,N})|^{2}\upsilon_{q}(X_{s}^{k,N},s)  ds.
\end{align*}

Then from   relation between the Stratonovich  and  Ito integral, we can deduce (\ref{cov1}).  Using similar  arguments we deduce 
 (\ref{cov2}) and (\ref{cov3}).
\end{proof}
We use the following notations
\begin{align}\label{DNdN}
B_{N}(dx)=&\varrho(x)dx- d( S_{t}^{N}(x))\nonumber\\
b_{N}(x)=&\varrho(x)-\left( S_{t}^{N}\ast\phi_{N}^{r}\right)(x)\\
\digamma_{N}^{k}(t)=&V_{t}^{k,N}-\upsilon(X_{t}^{k,N},t)\nonumber\\
\psi_{N}^{r}(x)=&N^{\gamma}\pi^{-d/2}e^{-|x|^{2}N^{2\gamma/d}}\nonumber
\end{align}
Note that
\begin{equation}\label{sigamNr}
    \psi_{N}^{r}\ast\psi_{N}^{r}=\psi_{N},
\end{equation}
\begin{align}\label{prophiN1}
    N^{2 \gamma / d} x_{i} \psi_{N}(x)=&-\nabla^{i} \psi_{N}(x),
\end{align}
\begin{align}\label{prophiN2}
    N^{4 \gamma / d} x_{i} x_{j} \psi_{N}(x)=& \zeta_{N, i j}(x)=\nabla^{i} \nabla^{j} \psi_{N}(x)+N^{2 \gamma / d} \delta_{i j} \psi_{N}(x)
\end{align}
and 
\begin{align}\label{prophiN3}
N^{6 \gamma / d} x_{i} x_{j} x_{q} \psi_{N}(x)=&-\nabla^{i} \nabla^{j} \nabla^{q} \psi_{N}(x)-N^{2 \gamma / d} \delta_{i j} \nabla^{q} \psi_{N}(x)-N^{2 \gamma / d} \delta_{i q} \nabla^{j} \psi_{N}(x) \nonumber\\
&-N^{2 \gamma / d} \delta_{j q} \nabla^{i} \psi_{N}(x), \quad i, j, q=1, \ldots, d.
\end{align}

\begin{lemma}\label{lemmaR1} We have

\[ 
R_{N, 1}\leq C (Q_{t}^{N} +  (Q_{t}^{N})^{3/2} N^{\gamma(d+4) / 2 d}).
\]
  
\end{lemma}

\begin{proof}

From \eqref{prophiN2} we obtain

\begin{align*}
R_{N, 1}=&-\frac{1}{N^{2}} \sum_{k, l=1}^{N} \zeta_{N, q q^{\prime}}\left(X_{t}^{k,N}-X_{t}^{l,N}\right)\left( \digamma_{N, q}^{k}(t)-\digamma_{N, q}^{l}(t)\right)\left( \digamma_{N, q^{\prime}}^{k}(t)-\digamma_{N, q^{\prime}}^{l}(t)\right)\nonumber\\
=&-\frac{2}{N^{2}} \sum_{k, l=1}^{N} \nabla^{q} \nabla^{q^{\prime}} \psi_{N}\left(X_{t}^{k,N}-X_{t}^{l,N}\right) \digamma_{N, q}^{k}(t)  \digamma_{N, q^{\prime}}^{k}(t)\nonumber \\
&+\frac{2}{N^{2}} \sum_{k, l=1}^{N} \nabla^{q} \nabla^{q^{\prime}} \psi_{N}\left(X_{t}^{k,N}-X_{t}^{l,N}\right)\digamma_{N, q}^{k}(t)  \digamma_{N, q^{\prime}}^{l}(t) \nonumber\\
&-\frac{N^{2 \gamma / d}}{N^{2}} \sum_{k, l=1}^{N} \psi_{N}\left(X_{t}^{k,N}-X_{t}^{l,N}\right)\left|\digamma_{N}^{k}(t)-  \digamma_{N}^{l}(t)\right|^{2}.
\end{align*}
We observe that 
\begin{align*}
 \nabla^{q}\nabla^{q^{\prime}}\psi_{N}(x-y)=&
    \nabla^{q}\nabla^{q^{\prime}}(\psi_{N}^{r}\ast\psi_{N}^{r}) (x-y)\\
    =&\nabla^{q}\nabla^{q^{\prime}} \int_{\R^{d}}  \psi_{N}^{r}(x-y-z) \ \psi_{N}^{r}(z) \ dz \\
=&\nabla^{q}\nabla^{q^{\prime}} \int_{\R^{d}}  \psi_{N}^{r}(x-z)\psi_{N}^{r}(z-y) \ dz\\
=&\int_{\R^{d}}  \nabla^{q}\psi_{N}^{r}(x-z) \nabla^{q^{\prime}}\psi_{N}^{r}(z-y) dz \\
=&- \int_{\R^{d}}\nabla^{q}\psi_{N}^{r}(z-x)\nabla^{q}\psi_{N}^{r}(z-y)dz.
\end{align*}
Thus we have 
\begin{align}\label{reescriRN0}
&\frac{2}{N^{2}} \sum_{k, l=1}^{N} \nabla^{q} \nabla^{q^{\prime}} \psi_{N}\left(X_{t}^{k,N}-X_{t}^{l,N}\right)\digamma_{N, q}^{k}(t)  \digamma_{N, q^{\prime}}^{l}(t)\nonumber\\
=&-\frac{2}{N^{2}} \int_{\R^{d}}\sum_{k, l=1}^{N}\nabla^{q}\psi_{N}^{r}(z-X_{t}^{k,N})\nabla^{q}\psi_{N}^{r}\left(z-X_{t}^{l,N}\right)\digamma_{N, q}^{k}(t)  \digamma_{N, q^{\prime}}^{l}(t)dz \nonumber\\
=&-\frac{2}{N^{2}} \int_{\R^{d}}\left(\sum_{k=1}^{N}\nabla^{q}\psi_{N}^{r}(z-X_{t}^{k,N})\digamma_{N, q}^{k}(t)\right)^{2}dz<0.
\end{align}
From  \eqref{reescriRN0} we arrive ate 
\begin{align}\label{reescriRN1}
R_{N, 1}\leq&-\frac{2}{N^{2}} \sum_{k, l=1}^{N} \nabla^{q} \nabla^{q^{\prime}} \psi_{N}\left(X_{t}^{k,N}-X_{t}^{l,N}\right) \digamma_{N, q}^{k}(t)  \digamma_{N, q^{\prime}}^{k}(t)\nonumber \\
&-\frac{N^{2 \gamma / d}}{N^{2}} \sum_{k, l=1}^{N} \psi_{N}\left(X_{t}^{k,N}-X_{t}^{l,N}\right)\left|\digamma_{N}^{k}(t)-  \digamma_{N}^{l}(t)\right|^{2}\nonumber\\
\leq&-\frac{2}{N^{2}} \sum_{k, l=1}^{N} \nabla^{q} \nabla^{q^{\prime}} \psi_{N}\left(X_{t}^{k,N}-X_{t}^{l,N}\right) \digamma_{N, q}^{k}(t)  \digamma_{N, q^{\prime}}^{k}(t).
\end{align}

From \eqref{defzetamatriz}, \eqref{defexp},  \eqref{limite inferiorphi1r}  and using Fourier transform 
we deduce  
\begin{eqnarray}\label{gradientedobleestimatiDN}
&&\left\|\nabla^{q} \nabla^{q^{\prime}}\left(B_{N}\ast\psi_{N}\right)\right\|_{\infty}\nonumber\\
&&\leq C \int_{\mathbb{R}^{d}} \left|\widetilde{B_{N}}(\lambda)\right| \lambda^{2} \exp \left(-\lambda^{2} / 2 N^{2 \gamma / d}\right)d \lambda\nonumber\\
&&\leq C\left(\int_{\mathbb{R}^{d}}\left|\widetilde{B_{N}}(\lambda)\right|^{2} \exp \left(-\lambda^{2} / 2 N^{2 \gamma / d}\right)d \lambda\right)^{1 / 2}\left(\int_{\mathbb{R}^{d}} \lambda^{4} \exp \left(-\lambda^{2} / 2 N^{2 \gamma / d}\right)d \lambda\right)^{1 / 2}\nonumber\\
&&\leq C\left(\int_{\mathbb{R}^{d}}\left|\widetilde{B_{N}}(\lambda)\right|^{2} \exp \left(-\lambda^{2} / 2 N^{2 \beta / d}\right)d \lambda\right)^{1 / 2} N^{\gamma(d+4) / 2 d}\nonumber\\
&&\leq C\left\|B_{N}\ast\phi_{N}^{r}\right\|_{2} N^{\gamma(d+4) / 2 d} \nonumber\\
&&\leq C N^{\gamma(d+4) / 2 d}\left(\left\|b_{N}\right\|_{2}+\left\|\varrho-\varrho \ast \phi_{N}^{r}\right\|_{2}\right)\nonumber\\
&&\leq C N^{\gamma(d+4) / 2 d}\left(\left\|b_{N}\right\|_{2}+N^{-\beta / d}\right)\nonumber\\
&&\leq C\left(N^{\gamma(d+4) / 2 d}\left\|b_{N}\right\|_{2}+1\right).
\end{eqnarray}

Using the same arguments we can obtain

\begin{equation}
    \left\|B_{N}\ast \psi_{N}\right\|_{\infty}\leq C\left(N^{\gamma/2}\left\|b_{N}\right\|_{2}+1\right).
\end{equation}
 
From  \eqref{gradientedobleestimatiDN} and the boundedness of 
 $\upsilon$, and  $\varrho$  we have

 \begin{eqnarray}\label{primertermoRN00}
&&\left|\frac{2}{N^{2}} \sum_{k, l=1}^{N} \nabla^{q} \nabla^{q^{\prime}} \psi_{N}\left(X_{t}^{k,N}-X_{t}^{l,N}\right) \digamma_{N, q}^{k}(t)  \digamma_{N, q^{\prime}}^{k}(t)\right|\nonumber\\
&&\quad=\left|\frac{2}{N} \sum_{k=1}^{N} \int_{\mathbb{R}^{d}}\nabla^{q} \nabla^{q^{\prime}} \psi_{N}(X_{t}^{k,N}-y)  \digamma_{N, q}^{k}(t)  \digamma_{N, q^{\prime}}^{k}(t)\{\varrho(y)(dy)-B_{N}(dy)\}  \right|\nonumber\\
&&\quad\leq\left| \frac{2}{N} \sum_{k=1}^{N}  \nabla^{q} \nabla^{q^{\prime}}\left(\varrho\ast\psi_{N}\right)(X_{t}^{k,N})  \digamma_{N, q}^{k}(t)  \digamma_{N, q^{\prime}}^{k}(t) \right|\nonumber\\
&&\qquad+\left| \frac{2}{N} \sum_{k=1}^{N}  \nabla^{q} \nabla^{q^{\prime}}\left(B_{N}\ast\psi_{N}\right)(X_{t}^{k,N})  \digamma_{N, q}^{k}(t)  \digamma_{N, q^{\prime}}^{k}(t)\right|\nonumber\\
&&\quad\leq \frac{2}{N} \sum_{k=1}^{N} \left|\digamma_{N, q}^{k}(t)\left\|\digamma_{N, q^{\prime}}^{k}(t)\right\| \nabla^{q} \nabla^{q^{\prime}}\left(\varrho\ast\psi_{N}\right)(X_{t}^{k,N}) \right| \nonumber\\
&&\qquad+ \frac{2}{N} \sum_{k=1}^{N}\left|\digamma_{N, q}^{k}(t)\left\|\digamma_{N, q^{\prime}}^{k}(t)\right\| \nabla^{q} \nabla^{q^{\prime}}\left(B_{N}\ast\psi_{N}\right)(X_{t}^{k,N}) \right| \nonumber\\
&&\quad\leq  C\frac{2}{N} \sum_{k=1}^{N}\left|\digamma_{N}^{k}(t)\right|^{2}\left(N^{\gamma(d+4) / 2 d}\left\|b_{N}\right\|_{2}+1\right)\nonumber\\
&&\quad\leq C (Q_{t}^{N} +  (Q_{t}^{N})^{3/2} N^{\gamma(d+4) / 2 d}).
\end{eqnarray}

\end{proof}

\begin{lemma}\label{lemmaR2} We have

\[
\left|R_{N, 2}\right| \leq C\left(Q_{t}^{N}+N^{-2 \beta / d} N^{2 \gamma(d+2) / d}+(Q_{t}^{N})^{3 / 2} N^{\gamma(d+4) / 2 d}\right).
\]

\end{lemma}

\begin{proof}
 We rewrite $R_{N,2}$
\begin{align}\label{R_{N,2}res}
  R_{N,2}
  =& \dfrac{2}{N}\sum_{k=1}^{N}\int_{\mathbb{R}^{d}}(\varrho(y, t)-\left( S_{t}^{N}\ast\phi_{N}^{r}\right)(y) )\left(V_{t}^{k,N}-\upsilon(X_{t}^{k,N},t)\right)\cdot\zeta_{N}(X_{t}^{k,N}-y)\nonumber\\
  &\qquad\qquad \qquad\qquad \qquad\qquad\qquad\qquad \qquad\qquad\left(\upsilon(X_{t}^{k,N}, t)-\upsilon(y, t)\right)dy\nonumber\\
  &-\frac{2}{N}\sum_{k=1}^{N}\int_{\R^{d}}\left(V_{t}^{k,N}-\upsilon(X_{t}^{k,N},t)\right)\cdot\zeta_{N}(X_{t}^{k,N}-y)\left(\upsilon(X_{t}^{k,N},t)-\upsilon(y,t)\right)\nonumber\\
  &\qquad \qquad\qquad\qquad \qquad\qquad\qquad \qquad\qquad\left\{ S_{t}^{N}(d y)-\left( S_{t}^{N}\ast\phi_{N}^{r}\right)(y) d y\right\}\nonumber\\  
  =&-\frac{2}{N} \sum_{k=1}^{N} \int_{\mathbb{R}^{d}}\zeta_{N, q q^{\prime}}(X_{t}^{k,N}-y) \digamma_{N, q}^{k}(t)\left(\upsilon_{q^{\prime}}(X_{N}^{t}(t))-\upsilon_{q^{\prime}}(y)\right)\nonumber\\
  &\qquad\qquad \qquad\qquad\qquad\qquad\qquad \qquad\qquad\left\{ S_{t}^{N}(d y)-\left( S_{t}^{N}\ast\phi_{N}^{r}\right)(y) d y\right\}\nonumber\\
  &+\frac{2}{N} \sum_{k=1}^{N}\digamma_{N, q}^{k}(t) \int_{\mathbb{R}^{d}} b_{N}(y) \zeta_{N, q q^{\prime}}(X_{t}^{k,N}-y)(\upsilon_{q^{\prime}}(X_{t}^{k,N},t)-\upsilon_{q^{\prime}}(y,t))dy.\nonumber\\
\end{align}
We will  estimate the first term of \eqref{R_{N,2}res}. 
 We observe that derivative  of the function  ${y\to\zeta_{N, q q^{\prime}}(x-y) \left(\upsilon_{q^{\prime}}(x,t)-\upsilon_{q^{\prime}}(y,t)\right)}$ is  uniformly bounded by  $CN^{\gamma(d+2) / d}$, it  is can be deduce from \eqref{defzetamatriz},  \eqref{defexp} and the regularity of $\upsilon$  and  $\varrho$. Then we have 
 \begin{align}
&\left| \frac{2}{N} \sum_{k=1}^{N} \int_{\mathbb{R}^{d}}\zeta_{N, q q^{\prime}}(X_{t}^{k,N}-y) \digamma_{N, q}^{k}(t)\left(\upsilon_{q^{\prime}}(X_{N}^{t}(t))-\upsilon_{q^{\prime}}(y)\right)\right.\nonumber\\
&\qquad\qquad\qquad\qquad\qquad\qquad\qquad\qquad\qquad\qquad\left\{ S_{t}^{N}(d y)-\left( S_{t}^{N}\ast\phi_{N}^{r}\right)(y) d y\right\}\Bigg|\nonumber\\
\leq& \frac{2}{N^{2}} \sum_{k,l=1}^{N} | \zeta_{N, q q^{\prime}}(X_{t}^{k,N}-X_{t}^{l,N}) \digamma_{N, q}^{k}(t)\left(\upsilon_{q^{\prime}}(X_{t}^{k,N},t)-\upsilon_{q^{\prime}}(X_{t}^{l,N},t)\right)\nonumber \\
&-\int_{\mathbb{R}^{d}} \phi_{N}^{r}(z-X_{t}^{l,N}) \zeta_{N, q q^{\prime}}(X_{t}^{k,N}-z) \digamma_{N, q}^{k}(t)\left(\upsilon_{q^{\prime}}(X_{t}^{k,N},t)-\upsilon_{q^{\prime}}(z,t)\right) dz|\nonumber \\
\leq& \frac{2}{N^{2}} \sum_{k,l=1}^{N} \left|\digamma_{N, q}^{k}(t)\right|\int_{\mathbb{R}^{d}} \phi_{N}^{r}(z-X_{t}^{l,N}) \left|  \zeta_{N, q q^{\prime}}(X_{t}^{k,N}-X_{t}^{l,N}) \left(\upsilon_{q^{\prime}}(X_{t}^{k,N},t)-\upsilon_{q^{\prime}}(X_{t}^{l,N},t)\right)\right.\nonumber\\
&\quad\left.-\zeta_{N, q q^{\prime}}(X_{t}^{k,N}-z)\left(\upsilon_{q^{\prime}}(X_{t}^{k,N},t)-\upsilon_{q^{\prime}}(z,t)\right)\right|dz \nonumber\\
\leq& \frac{2}{N^{2}} \sum_{k,l=1}^{N} \left|\digamma_{N, q}^{k}(t)\right|CN^{\gamma(d+2) / d}\int_{\mathbb{R}^{d}} |z-X_{t}^{l,N}|\phi_{N}^{r}(z-X_{t}^{l,N})dz \nonumber
\end{align}
\begin{align}\label{RN20}
\leq&  CN^{\gamma(d+2) / d} \frac{2}{N} \sum_{k=1}^{N} \left|\digamma_{N,q}^{k}(t)\right|\int_{\R^{d}}|w|\phi_{N}^{r}(w)dw\nonumber\qquad\qquad\qquad\qquad\qquad\qquad\qquad\qquad\\
\leq& \frac{1}{N}\sum_{k=1}^{N}\left( C^{2}\left|\digamma_{N,q}^{k}(t)\right|^{2} +\left(N^{\gamma(d+2) / d}\int_{\R^{d}}|w|\phi_{N}^{r}(w)dw\right)^{2}\right)\nonumber\\
\leq&  \dfrac{C}{N} \sum_{k=1}^{N} \left|\digamma_{N}^{k}(t)\right|^{2}+N^{-2 \beta / d} N^{2 \gamma(d+2) / d} .
\end{align}

We will  estimate the second term of \eqref{R_{N,2}res}. Adding and subtracting the term

\begin{align}\label{Restrela}
R_{N, q q^{\prime} q^{\prime \prime}}^{*}:=\frac{1}{N} \sum_{k=1}^{N}\digamma_{N, q}^{k}(t) \int_{\mathbb{R}^{d}}  b_{N}(y) \zeta_{N, q q^{\prime}}\left(X_{t}^{k,N}-y\right)\left(X_{t,q^{\prime \prime}}^{k,N}-y_{q^{\prime \prime}}\right) \nabla^{q^{\prime \prime}} \upsilon_{q^{\prime}}(y,t)dy,\nonumber\\
\end{align} 
we obtain that the second term
of \eqref{R_{N,2}res} is equal to 
\begin{align}\label{segutermoRN2}
 &\frac{1}{N} \sum_{k=1}^{N}\digamma_{N, q}^{k}(t) \int_{\mathbb{R}^{d}} b_{N}(y) \zeta_{N, q q^{\prime}}(X_{t}^{k,N}-y)\nonumber\\
&\qquad\qquad\times\left\{\upsilon_{q^{\prime}}(X_{t}^{k,N},t)-\upsilon_{q^{\prime}}(y,t)-\sum_{q^{\prime \prime}=1}^{d}\left(X_{t,q^{\prime \prime}}^{k,N}-y_{q^{\prime \prime}}\right) \nabla^{q^{\prime \prime}} \upsilon_{q^{\prime}}(y,t)\right\}dy\nonumber\\ 
&+\sum_{q^{\prime \prime}=1}^{d}R_{N, q q^{\prime} q^{\prime \prime}}^{*}.
\end{align}

Now, we  will estimated the first term of (\ref{segutermoRN2}). Doing  Taylor expansion of second order in the spatial variable  for $\upsilon$ we have 
\begin{equation*}
\upsilon_{q^{\prime}}(X_{t}^{k,N},t)=\upsilon_{q^{\prime}}(y,t)+  D\upsilon_{q^{\prime}}(y,t) \left(X_{N}^{k}(t)-y\right) +
 D^{2}\upsilon_{q^{\prime}}(y + \theta X_{t}^{k,N},t ) \left(X_{N}^{k}(t)-y\right)^{2} \end{equation*}
with $0<\theta<1$. Thus by boundedness of $\upsilon$ and formula (\ref{defzetamatriz}) we obtain 
\begin{align}
&\left|\frac{1}{N} \sum_{k=1}^{N}\digamma_{N, q}^{k}(t) \int_{\mathbb{R}^{d}} b_{N}(y) \zeta_{N, q q^{\prime}}(X_{t}^{k,N}-y)\right.\nonumber\\
&\qquad\qquad\times\bigg.\left\{\upsilon_{q^{\prime}}(X_{t}^{k,N},t)-\upsilon_{q^{\prime}}(y,t)-\sum_{q^{\prime \prime}=1}^{d}\left(X_{N,q^{\prime \prime}}^{k}(t)-y_{q^{\prime \prime}}\right) \nabla^{q^{\prime \prime}} \upsilon_{q^{\prime}}(y,t)\right\}dy \bigg| \nonumber\\
\leq & C  \frac{1}{N} \sum_{k=1}^{N}\left|\digamma_{N, q}^{k}(t)\right| \int_{\mathbb{R}^{d}} \left|b_{N}(y)\right|\left|X_{t}^{k,N}-y\right|^{4} N^{4 \gamma / d} \psi_{N}\left(X_{t}^{k,N}-y\right) dy. \nonumber\\
\end{align}

 Now, we have 
 
\begin{align}\label{RN21}
& C  \frac{1}{N} \sum_{k=1}^{N}\left|\digamma_{N, q}^{k}(t)\right| \int_{\mathbb{R}^{d}} \left|b_{N}(y)\right|\left|X_{t}^{k,N}-y\right|^{4} N^{4 \gamma / d} \psi_{N}\left(X_{t}^{k,N}-y\right) dy\nonumber\\
\leq & C \frac{1}{N} \sum_{k=1}^{N}\left|\digamma_{N, q}^{k}(t)\right|^{2} \int_{\mathbb{R}^{d}} \left|X_{t}^{k,N}-y\right|^{4} N^{4 \gamma / d} \psi_{N}\left(X_{t}^{k,N}-y\right)dy\nonumber\\
&+C \frac{1}{N} \sum_{k=1}^{N} \int_{\mathbb{R}^{d}} \left|b_{N}(y)\right|^{2}\left|X_{t}^{k,N}-y\right|^{4} N^{4 \gamma / d} \psi_{N}\left(X_{t}^{k,N}-y\right)dy \nonumber\\
=& C \frac{1}{N} \sum_{k=1}^{N}\left|\digamma_{N, q}^{k}(t)\right|^{2} \int_{\mathbb{R}^{d}} \left|X_{t}^{k,N}-y\right|^{4} N^{4 \gamma / d} \psi_{N}\left(X_{t}^{k,N}-y\right)dy\nonumber\\
&+C \int_{\mathbb{R}^{d}} \int_{\mathbb{R}^{d}} \left|b_{N}(y)\right|^{2}\left|x-y\right|^{4} N^{4 \gamma / d} \psi_{N}\left(x-y\right)dy d( S_{t}^{N}(x)) \nonumber\\
\leq & C\left( \frac{1}{N} \sum_{k=1}^{N}\left|\digamma_{N, q}^{k}(t)\right|^{2}+\left\| S_{t}^{N} \ast \psi_{N} \ast \psi_{N}\right\|_{\infty}\left\|b_{N}\right\|_{2}^{2}\right)\nonumber\\
\leq& C\left( \frac{1}{N} \sum_{k=1}^{N}\left|\digamma_{N, q}^{k}(t)\right|^{2}+\left(\left\|\varrho \ast \psi_{N}\right\|_{\infty}+\left\|B_{N} \ast \psi_{N}\right\|_{\infty}\right)\left\|b_{N}\right\|_{2}^{2}\right)\nonumber\\
\leq& C\left( \frac{1}{N} \sum_{k=1}^{N}\left|\digamma_{N, q}^{k}(t)\right|^{2}+\left(N^{\gamma / 2}\left\|b_{N}\right\|_{2}+1\right)\left\|b_{N}\right\|_{2}^{2}\right).
\end{align}

where we used that {\color{red}{ $\| B_{N}\ast \psi_{N}\|_{\infty}\leq   C( N^{\gamma / 2}\ \| b_{N}\ast \|_{2} + 1)$}}.
We will  estimate the second term of  \eqref{segutermoRN2}. Multiplying 
  \eqref{prophiN2} by  $w_{q^{\prime \prime}}$ substituting in  \eqref{prophiN3} we obtain 
\begin{align*}
\zeta_{N, q q^{\prime}}(w) w_{q^{\prime \prime}}&=-N^{-2 \gamma / d} \nabla^{q} \nabla^{q^{\prime}} \nabla^{q^{\prime \prime}} \psi_{N}(w)-\delta_{q q^{\prime}} \nabla^{q^{\prime \prime}} \psi_{N}(w)-\delta_{q q^{\prime \prime}} \nabla^{q^{\prime}} \psi_{N}(w)\\
&-\delta_{q^{\prime} q^{\prime \prime}} \nabla^{q} \psi_{N}(w).
\end{align*}
Thus we have
\begin{align*}
&\zeta_{N, q q^{\prime}}(w) w_{q^{\prime \prime}} \\
=&  -N^{-2 \gamma / d} \left(\nabla^{q}\psi_{N}^{r}\ast\nabla^{q^{\prime}}\nabla^{q^{\prime \prime}} \psi_{N}\right)(w)-\delta_{q q^{\prime}} \left(\nabla^{q^{\prime \prime}} \psi_{N}^{r}\ast\psi_{N}^{r}\right)(w)\\
&-\delta_{q q^{\prime \prime}}\left( \nabla^{q^{\prime}} \psi_{N}^{r}\ast\psi_{N}^{r}\right)(w)-\delta_{q^{\prime} q^{\prime \prime}} \left(\nabla^{q}\psi_{N}^{r}\ast\psi_{N}^{r}\right)(w)\nonumber\\
=& -\int_{\mathbb{R}^{d}} \nabla^{q} \psi_{N}^{r}\left(x\right) \nabla^{q^{\prime \prime}}  N^{-2 \gamma / d} \nabla^{q^{\prime}} \nabla^{q^{\prime \prime}} \psi_{N}^{r}(w-x)dx-\delta_{q q^{\prime}} \int_{\mathbb{R}^{d}}\nabla^{q^{\prime\prime}} \psi_{N}^{r}\left(x\right)\psi_{N}^{r}(w-x)dx\\
&-\delta_{q q^{\prime\prime}} \int_{\mathbb{R}^{d}}\nabla^{q} \psi_{N}^{r}\left(x\right)\psi_{N}^{r}(w-x)dx-\delta_{q^{\prime} q^{\prime\prime}} \int_{\mathbb{R}^{d}}\nabla^{q} \psi_{N}^{r}\left(x\right)\psi_{N}^{r}(w-x)dx.
\end{align*}

Then we deduce

\begin{align*}
R_{N, q q^{\prime} q^{\prime \prime}}^{*}=& -\frac{1}{N} \sum_{k=1}^{N}\digamma_{N, q}^{k}(t) \int_{\mathbb{R}^{d}}  b_{N}(y)\int_{\mathbb{R}^{d}} \nabla^{q} \psi_{N}^{r}\left(x\right) \nabla^{q^{\prime \prime}}  N^{-2 \gamma / d}\\
&\quad\qquad\qquad\qquad\qquad\qquad \times\nabla^{q^{\prime}} \nabla^{q^{\prime \prime}} \psi_{N}^{r}(X_{t}^{k,N}-y-x)\nabla^{q^{\prime \prime}} \upsilon_{q^{\prime}}(y,t)dxdy\\
&-\frac{1}{N} \sum_{k=1}^{N}\digamma_{N, q}^{k}(t) \int_{\mathbb{R}^{d}}  b_{N}(y)\delta_{q q^{\prime}} \int_{\mathbb{R}^{d}}\nabla^{q^{\prime\prime}} \psi_{N}^{r}\left(x\right)\\
&\qquad\qquad\qquad\qquad\qquad\qquad\qquad\times\psi_{N}^{r}(X_{t}^{k,N}-y-x)\nabla^{q^{\prime \prime}} \upsilon_{q^{\prime}}(y,t)dxdy\\
&-\frac{1}{N} \sum_{k=1}^{N}\digamma_{N, q}^{k}(t) \int_{\mathbb{R}^{d}}  b_{N}(y) \delta_{q q^{\prime\prime}} \int_{\mathbb{R}^{d}}\nabla^{q} \psi_{N}^{r}\left(x\right)\\
&\qquad\qquad\qquad\qquad\qquad\qquad\qquad\times\psi_{N}^{r}(X_{t}^{k,N}-y-x)\nabla^{q^{\prime \prime}} \upsilon_{q^{\prime}}(y,t)dxdy\\
&-\frac{1}{N} \sum_{k=1}^{N}\digamma_{N, q}^{k}(t) \int_{\mathbb{R}^{d}}  b_{N}(y) \delta_{q^{\prime} q^{\prime\prime}} \int_{\mathbb{R}^{d}}\nabla^{q} \psi_{N}^{r}\left(x\right)\\
&\qquad\qquad\qquad\qquad\qquad\qquad\qquad\times \psi_{N}^{r}(X_{t}^{k,N}-y-x)\nabla^{q^{\prime \prime}} \upsilon_{q^{\prime}}(y,t)dxdy.  
\end{align*}

Doing the  change of variable  $z=-x+X_{t}^{k,N}$,  we get

\begin{align*}
R_{N, q q^{\prime} q^{\prime \prime}}^{*}=&-\frac{1}{N} \sum_{k=1}^{N}\digamma_{N, q}^{k}(t) \int_{\mathbb{R}^{d}}  b_{N}(y)\int_{\mathbb{R}^{d}} \nabla^{q} \psi_{N}^{r}\left(-(z-X_{t}^{k,N})\right)   N^{-2 \gamma / d} \\
&\quad\qquad\qquad\qquad\qquad\qquad\qquad\qquad\times\nabla^{q^{\prime}} \nabla^{q^{\prime \prime}} \psi_{N}^{r}(z-y)\nabla^{q^{\prime \prime}} \upsilon_{q^{\prime}}(y,t)dzdy\\
&-\frac{1}{N} \sum_{k=1}^{N}\digamma_{N, q}^{k}(t) \int_{\mathbb{R}^{d}}  b_{N}(y)\delta_{q q^{\prime}} \int_{\mathbb{R}^{d}}\nabla^{q^{\prime\prime}} \psi_{N}^{r}\left(-(z-X_{t}^{k,N})\right)\\
&\qquad\qquad\qquad\qquad\qquad\qquad\qquad\qquad\qquad\times\psi_{N}^{r}(z-y)\nabla^{q^{\prime \prime}} \upsilon_{q^{\prime}}(y,t)dzdy\\
&-\frac{1}{N} \sum_{k=1}^{N}\digamma_{N, q}^{k}(t) \int_{\mathbb{R}^{d}}  b_{N}(y) \delta_{q q^{\prime\prime}} \int_{\mathbb{R}^{d}}\nabla^{q} \psi_{N}^{r}\left(-(z-x_{t}^{k,N})\right)\\
&\qquad\qquad\qquad\qquad\qquad\qquad\qquad\qquad\qquad\times\psi_{N}^{r}(z-y)\nabla^{q^{\prime \prime}} \upsilon_{q^{\prime}}(y,t)dzdy\\
&-\frac{1}{N} \sum_{k=1}^{N}\digamma_{N, q}^{k}(t) \int_{\mathbb{R}^{d}}  b_{N}(y) \delta_{q^{\prime} q^{\prime\prime}} \int_{\mathbb{R}^{d}}\nabla^{q} \psi_{N}^{r}\left(-(z-X_{t}^{k,N})\right)\\
&\qquad\qquad\qquad\qquad\qquad\qquad\qquad\qquad\qquad\times\psi_{N}^{r}(z-y)\nabla^{q^{\prime \prime}} \upsilon_{q^{\prime}}(y,t)dzdy.  
\end{align*}

Since
\begin{equation*}
\nabla^{q}\psi_{N}^{r}(x)=-\nabla^{q}\psi_{N}^{r}(-x)\qquad q=1,\cdots, d,    
\end{equation*}

we have
\begin{align*}
R_{N, q q^{\prime} q^{\prime \prime}}^{*}=&\frac{1}{N} \sum_{k=1}^{N}\digamma_{N, q}^{k}(t) \int_{\mathbb{R}^{d}}  b_{N}(y)\int_{\mathbb{R}^{d}} \nabla^{q} \psi_{N}^{r}\left(z-X_{t}^{k,N}\right)   N^{-2 \gamma / d} \nabla^{q^{\prime}} \nabla^{q^{\prime \prime}} \psi_{N}^{r}(z-y)\\
&\qquad\qquad\qquad\qquad\qquad\qquad\qquad\qquad\qquad\qquad\qquad\times\nabla^{q^{\prime \prime}} \upsilon_{q^{\prime}}(y,t)dzdy\\
&+\frac{1}{N} \sum_{k=1}^{N}\digamma_{N, q}^{k}(t) \int_{\mathbb{R}^{d}}  b_{N}(y)\delta_{q q^{\prime}} \int_{\mathbb{R}^{d}}\nabla^{q^{\prime\prime}} \psi_{N}^{r}\left(z-X_{t}^{k,N}\right)\psi_{N}^{r}(z-y)\\
&\qquad\qquad\qquad\qquad\qquad\qquad\qquad\qquad\qquad\qquad\qquad\times\nabla^{q^{\prime \prime}} \upsilon_{q^{\prime}}(y,t)dzdy\\
&+\frac{1}{N} \sum_{k=1}^{N}\digamma_{N, q}^{k}(t) \int_{\mathbb{R}^{d}}  b_{N}(y) \delta_{q q^{\prime\prime}} \int_{\mathbb{R}^{d}}\nabla^{q} \psi_{N}^{r}\left(z-x_{t}^{k,N}\right)\psi_{N}^{r}(z-y)\\
&\qquad\qquad\qquad\qquad\qquad\qquad\qquad\qquad\qquad\qquad\qquad\times\nabla^{q^{\prime \prime}} \upsilon_{q^{\prime}}(y,t)dzdy\\
&+\frac{1}{N} \sum_{k=1}^{N}\digamma_{N, q}^{k}(t) \int_{\mathbb{R}^{d}}  b_{N}(y) \delta_{q^{\prime} q^{\prime\prime}} \int_{\mathbb{R}^{d}}\nabla^{q} \psi_{N}^{r}\left(z-X_{t}^{k,N}\right)\psi_{N}^{r}(z-y)\\
&\qquad\qquad\qquad\qquad\qquad\qquad\qquad\qquad\qquad\qquad\qquad\times\nabla^{q^{\prime \prime}} \upsilon_{q^{\prime}}(y,t)dzdy.  
\end{align*}

Therefore we obtain

\begin{eqnarray*}
\left|R_{N, q q^{\prime} q^{\prime \prime}}^{*}\right| &\leq& \left| \int_{\mathbb{R}^{d}} \left(\frac{1}{N} \sum_{k=1}^{N}\digamma_{N, q}^{k}(t)  \nabla^{q} \psi_{N}^{r}\left(z-X_{t}^{k,N}\right)\right)\right. \nonumber\\
&&\qquad\times \left.\left(\int_{\mathbb{R}^{d}} b_{N}(y) \nabla^{q^{\prime \prime}} \upsilon_{q^{\prime}}(y) N^{-2 \gamma / d} \nabla^{q^{\prime}} \nabla^{q^{\prime \prime}} \psi_{N}^{r}(z-y)dy\right)dz \right| \nonumber\\
&+& \left| \int_{\mathbb{R}^{d}} \left(\frac{1}{N} \sum_{k=1}^{N}\digamma_{N, q}^{k}(t)  \nabla^{q^{\prime\prime}} \psi_{N}^{r}\left(z-X_{t}^{k,N}\right)\right)\right.\nonumber\\
& &\qquad\left.\times\left(\int_{\mathbb{R}^{d}}  b_{N}(y) \nabla^{q^{\prime \prime}} \upsilon_{q^{\prime}}(y) \psi_{N}^{r}(z-y)dy\right) dz\right| \delta_{q q^{\prime}}\nonumber\\
&+&\left|\int_{\mathbb{R}^{d}}\left(\frac{1}{N} \sum_{k=1}^{N}\digamma_{N, q}^{k}(t)  \nabla^{q} \psi_{N}^{r}\left(z-X_{t}^{k,N}\right)\right)\right.\nonumber \\
&&\qquad\left.\times\left(\int_{\mathbb{R}^{d}}  b_{N}(y) \nabla^{q^{\prime \prime}} \upsilon_{q^{\prime}}(y) \psi_{N}^{r}(z-y)dy\right) dz\right| \delta_{q q^{\prime \prime}}\nonumber \\
&&+\left| \int_{\mathbb{R}^{d}} \left(\frac{1}{N} \sum_{k=1}^{N}\digamma_{N, q}^{k}(t)  \nabla^{q} \psi_{N}^{r}\left(z-X_{t}^{k,N}\right)\right)\right.\nonumber \\
&&\qquad\left.\times\left(\int_{\mathbb{R}^{d}} b_{N}(y) \nabla^{q^{\prime \prime}} \upsilon_{q^{\prime}}(y) \psi_{N}^{r}(z-y)dy\right) dz\right| \delta_{q^{\prime} q^{\prime \prime}}.
\end{eqnarray*}

By   Holder  inequality we have
\begin{align}\label{RN22}
\left|R_{N, q q^{\prime} q^{\prime \prime}}^{*}\right| \leq& \left( \int_{\mathbb{R}^{d}} \left(\frac{1}{N} \sum_{k=1}^{N}\digamma_{N, q}^{k}(t)  \nabla^{q} \psi_{N}^{r}\left(z-X_{t}^{k,N}\right)\right)^{2}dz\right)^{1/2} \nonumber\\
&\qquad\times \left(\int_{\mathbb{R}^{d}}\left(\int_{\mathbb{R}^{d}} b_{N}(y) \nabla^{q^{\prime \prime}} \upsilon_{q^{\prime}}(y) N^{-2 \gamma / d} \nabla^{q^{\prime}} \nabla^{q^{\prime \prime}} \psi_{N}^{r}(z-y)dy\right)^{2}dz \right)^{1/2} \nonumber\\
&+ \left( \int_{\mathbb{R}^{d}} \left(\frac{1}{N} \sum_{k=1}^{N}\digamma_{N, q}^{k}(t)  \nabla^{q^{\prime\prime}} \psi_{N}^{r}\left(z-X_{t}^{k,N}\right)\right)^{2}dz\right)^{1/2}\nonumber\\
&\qquad\times\left(\int_{\mathbb{R}^{d}}\left(\int_{\mathbb{R}^{d}}  b_{N}(y) \nabla^{q^{\prime \prime}} \upsilon_{q^{\prime}}(y) \psi_{N}^{r}(z-y)dy\right)^{2} dz\right)^{1/2} \delta_{q q^{\prime}}\nonumber\\
&+\left(\int_{\mathbb{R}^{d}}\left(\frac{1}{N} \sum_{k=1}^{N}\digamma_{N, q}^{k}(t)  \nabla^{q} \psi_{N}^{r}\left(z-X_{t}^{k,N}\right)\right)^{2}dz\right)^{1/2}\nonumber \\
&\qquad\times\left(\int_{\mathbb{R}^{d}}\left(\int_{\mathbb{R}^{d}}  b_{N}(y) \nabla^{q^{\prime \prime}} \upsilon_{q^{\prime}}(y) \psi_{N}^{r}(z-y)dy\right)^{2} dz\right)^{1/2} \delta_{q q^{\prime \prime}}\nonumber \\
&+\left(\int_{\mathbb{R}^{d}} \left(\frac{1}{N} \sum_{k=1}^{N}\digamma_{N, q}^{k}(t)  \nabla^{q} \psi_{N}^{r}\left(z-X_{t}^{k,N}\right)\right)^{2}dz\right)^{1/2}\nonumber \\
&\qquad\times\left(\int_{\mathbb{R}^{d}}\left(\int_{\mathbb{R}^{d}} b_{N}(y) \nabla^{q^{\prime \prime}} \upsilon_{q^{\prime}}(y) \psi_{N}^{r}(z-y)dy\right)^{2} dz\right)^{1/2} \delta_{q^{\prime} q^{\prime \prime}}.
\end{align}

Using convolution inequality   and that  $\nabla \upsilon$
is bounded  we obtain

\begin{equation}\label{RN23}
 \int_{\mathbb{R}^{d}} \left(\int_{\mathbb{R}^{d}} b_{N}(y) \nabla^{q^{\prime \prime}} \upsilon_{q^{\prime}}(y) \psi_{N}^{r}(z-y)dy\right)^{2}dz\leq C\vert|b_{N}\vert|_{L^{2}}^{2}.  
\end{equation}

Using convolution inequality  and  $N^{-2 \gamma / d}\left|\nabla^{q^{\prime}} \nabla^{q^{\prime \prime}} \psi_{N}^{r}(w)\right| \leq C \psi_{N}(w)$,
 we have 

\begin{equation}\label{RN24}
\int_{\mathbb{R}^{d}}\left(\int_{\mathbb{R}^{d}}  b_{N}(y) \nabla^{q^{\prime \prime}} \upsilon_{q^{\prime}}(y) N^{-2 \gamma / d} \nabla^{q^{\prime}} \nabla^{q^{\prime \prime}} \psi_{N}^{r}(z-y)dy\right)^{2} dz\leq C\left\|b_{N}\right\|_{2}^{2}.    
\end{equation}

From  \eqref{sigamNr} and Young inequality, we have

\begin{eqnarray*}
&&\int_{\mathbb{R}^{d}}  \left(\frac{1}{N} \sum_{k=1}^{N}\digamma_{N, q}^{k}(t)  \nabla^{q^{\prime\prime}} \psi_{N}^{r}\left(z-X_{t}^{k,N}\right)\right)^{2} dz\nonumber\\
&&=-\frac{1}{N^{2}} \sum_{k,l=1}^{N}\digamma_{N, q}^{k}(t) \digamma_{N, q}^{l}(t) \nabla^{q^{\prime \prime}} \nabla^{q^{\prime \prime}} \psi_{N}\left(X_{t}^{k,N}-X_{t}^{l,N}\right)\nonumber\\
&&\leq\frac{1}{N^{2}} \sum_{k,l=1}^{N}\left|\digamma_{N, q}^{k}(t) \right|\left|\digamma_{N, q}^{l}(t)\right|\left| \nabla^{q^{\prime \prime}} \nabla^{q^{\prime \prime}} \psi_{N}\left(X_{t}^{k,N}-X_{t}^{l,N}\right)\right|\nonumber\\
&&\leq\frac{1}{2N^{2}} \sum_{k,l=1}^{N}\left(\left|\digamma_{N, q}^{k}(t) \right|^{2}+\left|\digamma_{N, q}^{l}(t)\right|^{2}\right)\left| \nabla^{q^{\prime \prime}} \nabla^{q^{\prime \prime}} \psi_{N}\left(X_{t}^{k,N}-X_{t}^{l,N}\right)\right|.
\end{eqnarray*}

From \eqref{DNdN} and 
$ \nabla^{q}\nabla^{q}\psi_{N}(x)=\nabla^{q}\nabla^{q}\psi_{N}(-x)\qquad q=1,\cdots, d$     
we obtain 
\begin{eqnarray*}
&&\int_{\mathbb{R}^{d}}  \left(\frac{1}{N} \sum_{k=1}^{N}\digamma_{N, q}^{k}(t)  \nabla^{q^{\prime\prime}} \psi_{N}^{r}\left(z-X_{t}^{k,N}\right)\right)^{2} dz\nonumber\\
&&\leq\frac{1}{N^{2}} \sum_{k,l=1}^{N}\left|\digamma_{N, q}^{k}(t) \right|^{2}\left| \nabla^{q^{\prime \prime}} \nabla^{q^{\prime \prime}} \psi_{N}\left(X_{t}^{k,N}-X_{t}^{l,N}\right)\right|\nonumber\\
&&=\frac{1}{N^{2}} \sum_{k=1}^{N}\left|\digamma_{N, q}^{k}(t) \right|^{2}\left|\int_{\R^{d}} \nabla^{q^{\prime \prime}} \nabla^{q^{\prime \prime}} \psi_{N}\left(X_{t}^{k,N}-y\right)\{\varrho(y)(dy)-B_{N}(dy)\}\right|\nonumber\\
&&\leq\frac{1}{N^{2}} \sum_{k=1}^{N}\left|\digamma_{N, q}^{k}(t) \right|^{2}\left|\nabla^{q^{\prime \prime}} \nabla^{q^{\prime \prime}} (\varrho\ast\psi_{N})\left(X_{t}^{k,N}\right)\right|\nonumber\\
&&\quad+\frac{1}{N^{2}} \sum_{k=1}^{N}\left|\digamma_{N, q}^{k}(t) \right|^{2}\left|\nabla^{q^{\prime \prime}} \nabla^{q^{\prime \prime}} (B_{N}\ast\psi_{N})\left(X_{t}^{k,N}\right)\right|.
\end{eqnarray*}

From \eqref{gradientedobleestimatiDN}, we have
\begin{align}\label{RN25}
\int_{\mathbb{R}^{d}}  \left(\frac{1}{N} \sum_{k=1}^{N}\digamma_{N, q}^{k}(t)  \nabla^{q^{\prime\prime}} \psi_{N}^{r}\left(z-X_{t}^{k,N}\right)\right)^{2} dz\leq  \frac{C}{N} \sum_{k=1}^{N}\left|\digamma_{N}^{k}(t)\right|^{2}\left(N^{\gamma(d+4) / 2 d}\left\|b_{N}\right\|_{2}+1\right).
\end{align}
From \eqref{RN22}-\eqref{RN25} we  deduce
\begin{eqnarray}\label{estimativaRNestrela}
\left|R_{N, q q^{\prime} q^{\prime \prime}}^{*}\right| \leq C\left\|b_{N}\right\|_{2}\left(\frac{1}{N} \sum_{k=1}^{N}\left|\digamma_{N}^{k}(t)\right|^{2}\left(N^{\gamma(d+4) / 2 d}\left\|b_{N}\right\|_{2}+1\right)\right)^{1/2}.
\end{eqnarray}

Therefore from  \eqref{RN20}, \eqref{RN21} e \eqref{estimativaRNestrela}, we obtain

\begin{align}
\left|R_{N, 2}\right| &\leq C\left\{\frac{1}{N} \sum_{k=1}^{N}\left|\digamma_{N}^{k}(t)\right|^{2} +N^{-2 \beta / d} N^{2 \gamma(d+2) / d}+\left\|b_{N}\right\|_{2}^{2}+\left\|b_{N}\right\|_{2}^{3} N^{\gamma / 2}\right.\nonumber \\
&\quad+\left\|b_{N}\right\|_{2}\left(\left.\frac{1}{N} \sum_{k=1}^{N}\left|\digamma_{N}^{k}(t)\right|^{2}\left(N^{\gamma(d+4) / 2 d}\left\|b_{N}\right\|_{2}+1\right)\right)^{1 / 2}\right\}. \nonumber\\
\end{align}

We observe that 
\begin{align*}
 \frac{1}{N} \sum_{k=1}^{N}&\left|\digamma_{N}^{k}(t)\right|^{2} +N^{-2 \beta / d} N^{2 \gamma(d+2) / d}+\left\|b_{N}\right\|_{2}^{2}+\left\|b_{N}\right\|_{2}^{3} N^{\gamma / 2}\\
 &\leq C\left(Q_{t}^{N}+N^{-2 \beta / d} N^{2 \gamma(d+2) / d}+(Q_{t}^{N})^{3 / 2} N^{\gamma(d+4) / 2 d}\right), 
\end{align*}
and 
\begin{align*}
\left\|b_{N}\right\|_{2}
&\left(\frac{1}{N} \sum_{k=1}^{N}\left|\digamma_{N}^{k}(t)\right|^{2}\left(N^{\gamma(d+4) / 2 d} \left\|b_{N}\right\|_{2}+1\right)\right)^{1/2}\\
&\leq \left\|b_{N}\right\|_{2}
\left(\frac{1}{N} \sum_{k=1}^{N}\left|\digamma_{N}^{k}(t)\right|^{2}\left(N^{\gamma(d+4) / 2 d} \left\|b_{N}\right\|_{2}+1\right)^{2}\right)^{1/2}\\
&\leq 2 \left\|b_{N}\right\|_{2}
\left(\frac{1}{N} \sum_{k=1}^{N}\left|\digamma_{N}^{k}(t)\right|^{2}\left(N^{\gamma(d+4) / d} \left\|b_{N}\right\|_{2}^{2}+1\right)\right)^{1/2}\\
& \leq 2 \left\|b_{N}\right\|_{2} (Q_{t}^{N})^{1/2} + N^{\gamma(d+4) / 2 d} \left\|b_{N}\right\|_{2}  Q_{t}^{N}\\
& \leq 2 Q_{t}^{N}+ N^{\gamma(d+4) / 2 d}   (Q_{t}^{N})^{3/2}
\end{align*}
Then we conclude 
\begin{align*}
\left|R_{N, 2}\right|\leq  C\left(Q_{t}^{N}+N^{-2 \beta / d} N^{2 \gamma(d+2) / d}+(Q_{t}^{N})^{3 / 2} N^{\gamma(d+4) / 2 d}\right).     
\end{align*}
\end{proof}

\begin{lemma}\label{lemmaR3} We have

\begin{eqnarray}
        |R_{N,3}|
        &\leq&C\left(N^{-\gamma/d}+\frac{1}{N}\sum_{k=1}^{N}|V_{t}^{k,N}-\upsilon(X_{t}^{k,N},t)|^{2}\right).
    \end{eqnarray}

\end{lemma}

\begin{proof}

 Let 

\begin{eqnarray*}
    R:=\left[\frac{\delta d}{4\gamma}\right]. 
\end{eqnarray*}

Using Taylor expansions for   $\varrho$ and  $\upsilon$  and the definition of   $\zeta_N$ we have 

\begin{align*}
-\int_{R^{d}}  &\varrho(y, t)\zeta_{N, q q^{\prime}}(x-y)(\upsilon_{q^{\prime}}(y, t)-\upsilon_{q^{\prime}}(x, t))dy\\
&=\int_{\R^{d}}\left[\sum_{|\alpha|=0}^{2R+3}\frac{1}{\alpha !}(y-x)^{\alpha}\frac{ \partial^{\alpha}\varrho}{\partial x^{\alpha}}(x,t)+\sum_{|\alpha|=2R+4}R_{\varrho,\alpha}(y)(y-x)^{\alpha}\right]\zeta_{N, q q^{\prime}}(x-y)\\
&\qquad\qquad\quad\times\left[\sum_{|\beta|=1}^{2R+3}\frac{1}{\beta !}(y-x)^{\beta}\frac{ \partial^{\beta}\upsilon_{q^{\prime}}}{\partial x^{\beta}}(x,t)+\sum_{|\beta|=2R+4}R_{\upsilon,\beta}(y)(y-x)^{\beta}\biggl.\right]dy\\
&=\sum_{|\alpha|=1}^{2R+3}\frac{1}{\alpha!}\varrho(x,t)\frac{ \partial^{\alpha}\upsilon_{q^{\prime}}}{\partial x^{\alpha}}(x,t)h_{N,q,q^{\prime}}^{\alpha}+\sum_{|\alpha|,|\beta|=1}^{2R+3}\frac{1}{\alpha !\beta!}\frac{ \partial^{\alpha}\varrho}{\partial x^{\alpha}}(x,t)\frac{ \partial^{\beta}\upsilon_{q^{\prime}}}{\partial x^{\beta}}(x,t)h_{N,q,q^{\prime}}^{\alpha+\beta}\\
&\qquad+\sum_{|\alpha|=2R+4}\sum_{|\beta|=1}^{2R+3}\frac{1}{\beta !}\frac{\partial^{\beta}\upsilon_{q^{\prime}}}{\partial x^{\beta}}(x,t)\int_{\R^{d}}R_{\varrho,\alpha}(y)(y-x)^{\alpha+\beta}\zeta_{N, q q^{\prime}}(x-y)dy\\
&\qquad+\sum_{|\alpha|,|\beta|=2R+4}\int_{\R^{d}}R_{\varrho,\alpha}(y)R_{\upsilon,\beta}(y)(y-x)^{\alpha+\beta}\zeta_{N, q q^{\prime}}(x-y)dy\\
&\qquad+\sum_{|\alpha|=0}^{2R+3}\sum_{|\beta|=2R+4}\frac{1}{\alpha !}\frac{ \partial^{\alpha}\varrho}{\partial x^{\alpha}}(x,t)\int_{\R^{d}}R_{\upsilon,\beta}(y)(y-x)^{\alpha+\beta}\zeta_{N, q q^{\prime}}(x-y)dy,
\end{align*}

where 
\begin{align*}
h_{ N,q,q^{\prime}}^{\alpha}\\
=&\frac{N^{\gamma(d+4)/d}}{(2\pi)^{d/2}}\int_{\R^{d}}x^{\alpha}x_{q}x_{q^{\prime}}e^{-(N^{\gamma/d}|x|)^{2}/2}dx\\
=&N^{\gamma(d+4)/d}i^{-(|\alpha|+2)}\frac{\partial^{\alpha}}{\partial\xi^{\alpha}}\nabla^{q}\nabla^{q^{\prime}}\left(\frac{1}{N^{\gamma}}e^{-(N^{-\gamma/d}|\xi|)^{2}/2}\right)(0)\\
=&(-1)^{|\alpha|+1}i^{|\alpha|}N^{4\gamma/d}\frac{\partial^{\alpha}}{\partial\xi^{\alpha}}\nabla^{q}\nabla^{q^{\prime}}\left(e^{-(N^{-\gamma/d}|\xi|)^{2}/2}\right)(0)\\
=&(-1)^{|\alpha|+1}i^{|\alpha|}N^{4\gamma/d}\frac{\partial^{{\alpha}^{q,q^{\prime}}}}{\partial\xi^{{\alpha}^{q,q^{\prime}}}}\left(e^{-(N^{-\gamma/d}|\xi|)^{2}/2}\right)(0)\\
=&\left\{
\begin{array}{lc}
0\quad\qquad\qquad\qquad\qquad\qquad\qquad\qquad\qquad\quad\quad\quad\text{ se }\alpha_{i}^{q,q^{\prime}}\text{ is odd for some }i\\[0.3cm]
(-2)^{-|\alpha^{q,q^{\prime}|}/2}(-1)^{|\alpha|+1}i^{|\alpha|}N^{\gamma(4-|\alpha^{q,q^{\prime}}|)/d}\frac{\alpha^{q,q^{\prime}}!}{\left(\alpha^{q,q^{\prime}}/2\right)!}\quad\text{se }\alpha_{i}^{q,q^{\prime}}\text{ is  even for all } i
\end{array}
\right.
\end{align*}
with ${\alpha}^{q,q^{\prime}}=\left(\alpha_{1},\cdots,\alpha_{q-1},\alpha_{q}+1,\alpha_{q+1},\cdots,\alpha_{q^{\prime}-1},\alpha_{q^{\prime}}+1,\alpha_{q^{\prime}+1},\cdots,\alpha_{d}\right)$.\\

We observe  that for ${\alpha}^{q,q^{\prime}}\geq 4$ we have
\begin{equation*}
    |h_{N,q,q^{\prime}}^{\alpha}|\leq C N^{-\lambda/d}.
\end{equation*}
We have 
\begin{align*}
&\sum_{|\alpha|=1}^{2R+3}\frac{1}{\alpha!}\varrho(x,t)\frac{ \partial^{\alpha}\upsilon_{q^{\prime}}}{\partial x^{\alpha}}(x,t)h_{N,q,q^{\prime}}^{\alpha}=\sum_{|\alpha|=2}^{2R+3}\frac{1}{\alpha!}\varrho(x,t)\frac{ \partial^{\alpha}\upsilon_{q^{\prime}}}{\partial x^{\alpha}}(x,t)h_{N,q,q^{\prime}}^{\alpha}\\
=&\sum_{\stackrel{q^{\prime}=1}{q^{\prime}\neq q}}^{d}\varrho(x,t)\nabla^{q}\nabla^{q^{\prime}}\upsilon_{q^{\prime}}(x,t)+\sum_{\stackrel{l=1}{l\neq q}}^{d}\frac{1}{2}\varrho(x,t)\nabla^{l}\nabla^{l}\upsilon_{q}(x,t)+\frac{3}{2}\varrho(x,t)\nabla^{q}\nabla^{q}\upsilon_{q}(x,t)\\
&+\sum_{q^{\prime}=1}^{d}\sum_{|\alpha|=3}^{2R+3}\frac{1}{\alpha!}\varrho(x,t)\frac{ \partial^{\alpha}\upsilon_{q^{\prime}}}{x^{\alpha}}(x,t)h_{N,q,q^{\prime}}^{\alpha}\\
=&\sum_{\stackrel{q^{\prime}=1}{q^{\prime}\neq q}}^{d}\varrho(x,t)\nabla^{q}\nabla^{q^{\prime}}\upsilon_{q^{\prime}}(x,t)+\sum_{\stackrel{q^{\prime}=1}{q^{\prime}\neq q}}^{d}\frac{1}{2}\varrho(x,t)\nabla^{q^{\prime}}\nabla^{q^{\prime}}\upsilon_{q}(x,t)+\frac{3}{2}\varrho(x,t)\nabla^{q}\nabla^{q}\upsilon_{q}(x,t)\\
&+\sum_{q^{\prime}=1}^{d}\sum_{|\alpha|=3}^{2R+3}\frac{1}{\alpha!}\varrho(x,t)\frac{ \partial^{\alpha}\upsilon_{q^{\prime}}}{x^{\alpha}}(x,t)h_{N,q,q^{\prime}}^{\alpha}\\
=&\sum_{q^{\prime}=1}^{d}\varrho(x,t)\nabla^{q}\nabla^{q^{\prime}}\upsilon_{q^{\prime}}(x,t)+\sum_{q^{\prime}=1}^{d}\frac{1}{2}\varrho(x,t)\nabla^{q^{\prime}}\nabla^{q^{\prime}}\upsilon_{q}(x,t)\\
&+\sum_{q^{\prime}=1}^{d}\sum_{|\alpha|=3}^{2R+3}\frac{1}{\alpha!}\varrho(x,t)\frac{ \partial^{\alpha}\upsilon_{q^{\prime}}}{x^{\alpha}}(x,t)h_{N,q,q^{\prime}}^{\alpha}
\end{align*}

and

\begin{align*}
&\sum_{|\alpha|,|\beta|=1}^{2R+3}\frac{1}{\alpha !\beta!}\frac{ \partial^{\alpha}\varrho}{\partial x^{\alpha}}(x,t)\frac{ \partial^{\beta}\upsilon_{q}}{\partial  x^{\beta}}(x,t)h_{N,q,q^{\prime}}^{\alpha+\beta}\\
=&\sum_{\stackrel{q^{\prime}=1}{q^{\prime}\neq q}}^{d}\sum_{|\alpha|,|\beta|=1}\frac{1}{\alpha !\beta!}\frac{ \partial^{\alpha}\varrho}{\partial x^{\alpha}}(x,t)\frac{ \partial^{\beta}\upsilon_{q}}{\partial  x^{\beta}}(x,t)h_{N,q,q^{\prime}}^{\alpha+\beta}+\sum_{|\alpha|,|\beta|=1}\frac{1}{\alpha !\beta!}\frac{ \partial^{\alpha}\varrho}{\partial x^{\alpha}}(x,t)\frac{ \partial^{\beta}\upsilon_{q}}{\partial  x^{\beta}}(x,t)h_{N,q,q^{\prime}}^{\alpha+\beta}\\
&+\sum_{q^{\prime}=1}^{d}\sum_{|\alpha|,|\beta|=2}^{2R+3}\frac{1}{\alpha !\beta!}\frac{ \partial^{\alpha}\varrho}{\partial x^{\alpha}}(x,t)\frac{ \partial^{\beta}\upsilon_{q}}{\partial  x^{\beta}}(x,t)h_{N,q,q^{\prime}}^{\alpha+\beta}\\
=&\sum_{\stackrel{q^{\prime}=1}{q^{\prime}\neq q}}^{d}\nabla^{q}\varrho(x,t)\nabla^{q^{\prime}}\upsilon_{q^{\prime}}(x,t)+\sum_{\stackrel{q^{\prime}=1}{q^{\prime}\neq q}}^{d}\nabla^{q^{\prime}}\varrho(x,t)\nabla^{q}\upsilon_{q^{\prime}}(x,t)+\sum_{\stackrel{l=1}{l\neq q}}^{d}\nabla^{l}\varrho(x,t)\nabla^{l}\upsilon_{q}(x,t)\\
&+3\nabla^{q}\varrho(x,t)\nabla^{q}\upsilon_{q}(x,t)+\sum_{q^{\prime}=1}^{d}\sum_{|\alpha|,|\beta|=2}^{2R+3}\frac{1}{\alpha !\beta!}\frac{ \partial^{\alpha}\varrho}{\partial x^{\alpha}}(x,t)\frac{ \partial^{\beta}\upsilon_{q}}{\partial  x^{\beta}}(x,t)h_{N,q,q^{\prime}}^{\alpha+\beta}\\
\end{align*}
\begin{align*}
=&\sum_{q^{\prime}=1}^{d}\nabla^{q}\varrho(x,t)\nabla^{q^{\prime}}\upsilon_{q^{\prime}}(x,t)+\sum_{q^{\prime}=1}^{d}\nabla^{q^{\prime}}\varrho(x,t)\nabla^{q}\upsilon_{q^{\prime}}(x,t)+\sum_{q^{\prime}=1}^{d}\nabla^{q^{\prime}}\varrho(x,t)\nabla^{q^{\prime}}\upsilon_{q}(x,t)\\
&+\sum_{q^{\prime}=1}^{d}\sum_{|\alpha|,|\beta|=2}^{2R+3}\frac{1}{\alpha !\beta!}\frac{ \partial^{\alpha}\varrho}{\partial x^{\alpha}}(x,t)\frac{ \partial^{\beta}\upsilon_{q}}{\partial  x^{\beta}}(x,t)h_{N,q,q^{\prime}}^{\alpha+\beta}.
\end{align*}
Then we deduce
\begin{align}\label{2eqintegrmatri}
&-\int_{\mathbb{R}^{d}}  \varrho(y, t) \zeta_{N, q q^{\prime}}(x-y)\left(\upsilon_{q^{\prime}}(x, t)-\upsilon_{q^{\prime}}(y, t)\right)dy\nonumber\\
=&\nabla^{q}\varrho(x,t)\nabla^{q^{\prime}}\upsilon_{q^{\prime}}(x,t)+\frac{1}{2}\varrho(x,t)\nabla^{q}\nabla^{q^{\prime}}\upsilon_{q^{\prime}}(x,t)+\nabla^{q^{\prime}}\varrho(x,t)[\nabla^{q^{\prime}}\upsilon_{q}(x,t)+\nabla^{q}\upsilon_{q^{\prime}}(x,t)]\nonumber\\
&+\frac{1}{2}\varrho(x,t)\left[\nabla^{q^{\prime}}\nabla^{q^{\prime}}\upsilon_{q}(x,t)+\nabla^{q^{\prime}}\nabla^{q}\upsilon_{q^{\prime}}(x,t)\right]+\sum_{|\alpha|=3}^{2R+3}\frac{1}{\alpha!}\varrho(x,t)\frac{ \partial^{\alpha}\upsilon_{q^{\prime}}}{x^{\alpha}}(x,t)h_{N,q,q^{\prime}}^{\alpha}\nonumber\\
&+\sum_{|\alpha|,|\beta|=2}^{2R+3}\frac{1}{\alpha !\beta!}\frac{ \partial^{\alpha}\varrho}{\partial x^{\alpha}}(x,t)\frac{ \partial^{\beta}\upsilon_{q}}{\partial  x^{\beta}}(x,t)h_{N,q,q^{\prime}}^{\alpha+\beta}\nonumber\\
&+\sum_{|\alpha|=2R+4}\sum_{|\beta|=1}^{2R+3}\frac{1}{\beta !}\frac{\partial^{\beta}\upsilon_{q^{\prime}}}{\partial x^{\beta}}(x,t)\int_{\R^{d}}R_{\varrho,\alpha}(y)(y-x)^{\alpha+\beta}\zeta_{N, q q^{\prime}}(x-y)dy\nonumber\\
&+\sum_{|\alpha|,|\beta|=2R+4}\int_{\R^{d}}R_{\varrho,\alpha}(y)R_{\upsilon,\beta}(y)(y-x)^{\alpha+\beta}\zeta_{N, q q^{\prime}}(x-y)dy\nonumber\\
&+\sum_{|\alpha|=0}^{2R+3}\sum_{|\beta|=2R+4}\frac{1}{\alpha !}\frac{ \partial^{\alpha}\varrho}{\partial x^{\alpha}}(x,t)\int_{\R^{d}}R_{\upsilon,\beta}(y)(y-x)^{\alpha+\beta}\zeta_{N, q q^{\prime}}(x-y)dy\nonumber\\
=&\frac{1}{2\varrho(x,t)}\nabla^{q}\left(\varrho(x,t)^{2}\nabla^{q^{\prime}}\upsilon_{q^{\prime}}(x,t)\right)+\frac{1}{2\varrho(x,t)}\nabla^{q^{\prime}}\left(\varrho(x,t)^{2}\left[\nabla^{q^{\prime}}\upsilon_{q}(x,t)+\nabla^{q}\upsilon_{q^{\prime}}(x,t)\right]\right)\nonumber\\
&+\sum_{|\alpha|=3}^{2R+3}\frac{1}{\alpha!}\varrho(x,t)\frac{ \partial^{\alpha}\upsilon_{q^{\prime}}}{ \partial x^{\alpha}}(x,t)h_{N,q,q^{\prime}}^{\alpha}+\sum_{|\alpha|,|\beta|=2}^{2R+3}\frac{1}{\alpha !\beta!}\frac{ \partial^{\alpha}\varrho}{\partial x^{\alpha}}(x,t)\frac{ \partial^{\beta}\upsilon_{q}}{\partial  x^{\beta}}(x,t)h_{N,q,q^{\prime}}^{\alpha+\beta}\nonumber\\
&+\sum_{|\alpha|=2R+4}\sum_{|\beta|=1}^{2R+3}\frac{1}{\beta !}\frac{\partial^{\beta}\upsilon_{q^{\prime}}}{\partial x^{\beta}}(x,t)\int_{\R^{d}}R_{\varrho,\alpha}(y)(y-x)^{\alpha+\beta}\zeta_{N, q q^{\prime}}(x-y)dy\nonumber\\
&+\sum_{|\alpha|,|\beta|=2R+4}\int_{\R^{d}}R_{\varrho,\alpha}(y)R_{\upsilon,\beta}(y)(y-x)^{\alpha+\beta}\zeta_{N, q q^{\prime}}(x-y)dy\nonumber\\
&+\sum_{|\alpha|=0}^{2R+3}\sum_{|\beta|=2R+4}\frac{1}{\alpha !}\frac{ \partial^{\alpha}\varrho}{\partial x^{\alpha}}(x,t)\int_{\R^{d}}R_{\upsilon,\beta}(y)(y-x)^{\alpha+\beta}\zeta_{N, q q^{\prime}}(x-y)dy.
\end{align}

By \eqref{2eqintegrmatri} we deduce
\begin{align*}
&R_{N,3}\nonumber\\
=&\dfrac{2}{N}\sum_{k=1}^{N}(V_{t,q}^{k,N}-\upsilon_{q}(X_{t}^{k,N},t))\left(\sum_{|\alpha|=3}^{2R+3}\frac{1}{\alpha!}\varrho(X_{t}^{k,N},t)\frac{ \partial^{\alpha}\upsilon_{q^{\prime}}}{x^{\alpha}}(X_{t}^{k,N},t)h_{N,q,q^{\prime}}^{\alpha}\right.\\
&+\sum_{|\alpha|,|\beta|=2}^{2R+3}\frac{1}{\alpha !\beta!}\frac{ \partial^{\alpha}\varrho}{\partial x^{\alpha}}(X_{t}^{k,N},t)\frac{ \partial^{\beta}\upsilon_{q}}{\partial  x^{\beta}}(X_{t}^{k,N},t)h_{N,q,q^{\prime}}^{\alpha+\beta}\nonumber\\
&+\sum_{|\alpha|=2R+4}\sum_{|\beta|=1}^{2R+3}\frac{1}{\beta !}\frac{\partial^{\beta}\upsilon_{q^{\prime}}}{\partial x^{\beta}}(X_{t}^{k,N},t)\int_{\R^{d}}R_{\varrho,\alpha}(y)(y-X_{t}^{k,N})^{\alpha+\beta}\zeta_{N, q q^{\prime}}(X_{t}^{k,N}-y)dy\nonumber\\
&+\sum_{|\alpha|,|\beta|=2R+4}\int_{\R^{d}}R_{\varrho,\alpha}(y)R_{\upsilon,\beta}(y)(y-X_{t}^{k,N})^{\alpha+\beta}\zeta_{N, q q^{\prime}}(X_{t}^{k,N}-y)dy\nonumber\\
&\left.+\sum_{|\alpha|=0}^{2R+3}\sum_{|\beta|=2R+4}\frac{1}{\alpha !}\frac{ \partial^{\alpha}\varrho}{\partial x^{\alpha}}(X_{t}^{k,N},t)\int_{\R^{d}}R_{\upsilon,\beta}(y)(y-X_{t}^{k,N})^{\alpha+\beta}\zeta_{N, q q^{\prime}}(X_{t}^{k,N}-y)dy\right).
\end{align*}
By simple calculation we have 
\begin{align*}
 &\left|\int_{\R^{d}}R_{\varrho,\alpha}(y)(y-X_{t}^{k,N})^{\alpha+\beta}\zeta_{N, q q^{\prime}}(X_{t}^{k,N}-y)dy\right|\\
 &\leq\int_{\R^{d}}|R_{\varrho,\alpha}(y)||(y-X_{t}^{k,N})^{\alpha+\beta}\zeta_{N, q q^{\prime}}(X_{t}^{k,N}-y)|dy\\
&\leq C\int_{\R^{d}}|(y-X_{t}^{k,N})^{\alpha+\beta}\zeta_{N, q q^{\prime}}(X_{t}^{k,N}-y)|dy,
\end{align*}
\begin{align*}
 &\left|\int_{\R^{d}}R_{\upsilon,\beta}(y)(y-X_{t}^{k,N})^{\alpha+\beta}\zeta_{N, q q^{\prime}}(X_{t}^{k,N}-y)dy\right|\\
 &\leq\int_{\R^{d}}|R_{\upsilon,\beta}(y)||(y-X_{t}^{k,N})^{\alpha+\beta}\zeta_{N, q q^{\prime}}(X_{t}^{k,N}-y)|dy\\
 &\leq C\int_{\R^{d}}|(y-X_{t}^{k,N})^{\alpha+\beta}\zeta_{N, q q^{\prime}}(X_{t}^{k,N}-y)|dy,
\end{align*}
\begin{align*}
&\left|\int_{\R^{d}}R_{\varrho,\alpha}(y)R_{\upsilon,\beta}(y)(y-X_{t}^{k,N})^{\alpha+\beta}\zeta_{N, q q^{\prime}}(X_{t}^{k,N}-y)dy\right|\\
&\leq\int_{\R^{d}}|R_{\varrho,\alpha}(y)R_{\upsilon,\beta}(y)||(y-X_{t}^{k,N})^{\alpha+\beta}\zeta_{N, q q^{\prime}}(X_{t}^{k,N}-y)|dy\\
&\leq C\int_{\R^{d}}|(y-X_{t}^{k,N})^{\alpha+\beta}\zeta_{N, q q^{\prime}}(X_{t}^{k,N}-y)|dy.
\end{align*}

Doing   variable change we obtain 

\begin{eqnarray*}
    &&\int_{\R^{d}}|(y-X_{t}^{k,N})^{\alpha+\beta}\zeta_{N, q q^{\prime}}(X_{t}^{k,N}-y)|dy=\int_{\R^{d}}|y|^{\alpha+\beta}|\zeta_{N, q q^{\prime}}(y)|dy\\
    &&\qquad=N^{4\gamma/d}(2\pi)^{-d/2}\int_{\R^{d}}|y|^{\alpha+\beta}y_{q}y_{q^{\prime}}e^{-|N^{\gamma/d}y|^{2}}N^{\gamma}dy\\
    &&\qquad=N^{\gamma/d(2-|\alpha+\beta|)}(2\pi)^{-d/2}\int_{\R^{d}}|N^{\gamma/d}y|^{\alpha+\beta}N^{\gamma/d}y_{q}N^{\gamma/d}y_{q^{\prime}}e^{-|N^{\gamma/d}y|^{2}}N^{\gamma}dy\\
    &&\qquad=N^{\gamma(2-|\alpha+\beta|)/d}(2\pi)^{-d/2}\int_{\R^{d}}|x|^{\alpha+\beta}x_{q}x_{q^{\prime}}e^{-|x|^{2}}dx.
    \end{eqnarray*}

	Then we deduce 

		\begin{eqnarray}
        |R_{N,3}|&\leq& CN^{-\gamma/d}\dfrac{2}{N}\sum_{k=1}^{N}|V_{t}^{k,N}-\upsilon(X_{t}^{k,N},t)|\nonumber\\
        &\leq& CN^{-\gamma/d}\left(\sum_{k=1}^{N}\dfrac{|V_{t}^{k,N}-\upsilon(X_{t}^{k,N},t)|^{2}}{N}\right)^{1/2}\left(\sum_{k=1}^{N}\dfrac{1}{N}\right)^{1/2}\nonumber\\
        &\leq&CN^{-\gamma/d}\left(1+\frac{1}{N}\sum_{k=1}^{N}|V_{t}^{k,N}-\upsilon(X_{t}^{k,N},t)|^{2}\right)^{1/2}\nonumber\\
        &\leq&CN^{-\gamma/d}\left(1+\frac{1}{N}\sum_{k=1}^{N}|V_{t}^{k,N}-\upsilon(X_{t}^{k,N},t)|^{2}\right)\nonumber\\
        &\leq&C\left(N^{-\gamma/d}+\frac{1}{N}\sum_{k=1}^{N}|V_{t}^{k,N}-\upsilon(X_{t}^{k,N},t)|^{2}\right).
    \end{eqnarray}

\end{proof}

\section*{Acknowledgements}

Author Jesus Correa  has received research grants from CNPq
through the grant  $141464/2020-8$.Author Christian Olivera is partially supported by FAPESP by the grant  $2020/04426-6$,  by FAPESP-ANR by the grant Stochastic and Deterministic Analysis for Irregular Models$-2022/03379-0$ and  CNPq by the grant $422145/2023-8$.

\end{document}